\documentclass[ejs]{imsart}
\usepackage{amsthm,amsmath}
\RequirePackage[OT1]{fontenc}
\RequirePackage[numbers]{natbib}
\RequirePackage[colorlinks,citecolor=blue,urlcolor=blue]{hyperref}
\usepackage{mathtools, bm, bbold}
\usepackage{amssymb}
\usepackage{appendix}
\usepackage[ruled]{algorithm2e}
\usepackage{parskip}
\usepackage{enumitem}  
\usepackage{booktabs}
\usepackage{colortbl}
\usepackage{tabu}
\usepackage{float}

\doi{10.1214/154957804100000000}
\pubyear{0000}
\volume{0}
\firstpage{1}
\lastpage{1}

\startlocaldefs
        \newtheorem{theorem}{Theorem}[subsection]
        
        \newtheorem{lemma}[theorem]{Lemma}
        \newtheorem{definition}[theorem]{Definition}
        \newtheorem{example}[theorem]{Example}
        \newtheorem{property}[theorem]{Property}
        \newtheorem{remark}[theorem]{Remark}

        \makeatletter
        \newcommand{\@giventhatstar}[2]{\left(#1\;\middle|\;#2\right)}
        \newcommand{\@giventhatnostar}[3][]{#1(#2\;#1|\;#3#1)}
        \newcommand{\giventhat}{\@ifstar\@giventhatstar\@giventhatnostar}
        \makeatother

        \renewcommand{\phi}{\varphi}
        \newcommand{\indep}{\perp\!\!\!\perp}
        
        \newcommand{\norme}[1]{\lVert #1\rVert_{2}^{2}}
        \newcommand{\scalarprod}[2]{\,\langle #1,#2\rangle}
        \DeclareMathOperator*{\argmin}{arg\,min}
        \DeclareMathOperator*{\Times}{\bigtimes}
        \DeclareMathOperator*{\BigCup}{\bigcup}
        
        \newcommand{\pluseq}{\mathrel{{+}{=}}}
        
        \renewcommand{\epsilon}{\ensuremath{\varepsilon}}
        \SetKwInput{KwSideProd}{Side-product}
\endlocaldefs

\newenvironment{DIFnomarkup}{}{}

\begin{document}

\begin{frontmatter}

\title{Estimation of multivariate generalized gamma convolutions through Laguerre expansions.\thanksref{t1}}
\thankstext{t1}{This is an original paper}
\runtitle{Estimation of multivariate generalized gamma convolutions}

\author{\fnms{Oskar} \snm{Laverny}\corref{}\ead[label=e1]{oskar.laverny@univ-lyon1.fr}}
\address{Institut Camille Jordan, UMR 5208,\\
Université Claude Bernard Lyon 1, Lyon, France\\
SCOR SE\\\printead{e1}}
\author{\fnms{Esterina} \snm{Masiello}\ead[label=e2]{esterina.masiello@univ-lyon1.fr}}
\and
\author{\fnms{Véronique} \snm{Maume-Deschamps}\ead[label=e3]{veronique.maume@univ-lyon1.fr}}
\address{Institut Camille Jordan, UMR 5208,\\
Université Claude Bernard Lyon 1, Lyon, France\\
\printead{e2,e3}}
\author{\fnms{Didier} \snm{Rullière}\ead[label=e4]{didier.rulliere@emse.fr}}
\address{Mines Saint-Étienne,\\
UMR CNRS 6158, LIMOS, Saint-Étienne, France\\
\printead{e4}}
\runauthor{O. Laverny et al.}

\begin{abstract}
The generalized gamma convolutions class of distributions appeared in Thorin's work 
  while looking for the infinite divisibility of the log-Normal and Pareto distributions.
Although these distributions have been extensively studied in the univariate case,
  the multivariate case and the dependence structures that can arise from it have received little interest in the literature.
Furthermore,
  only one projection procedure for the univariate case was recently constructed,
  and no estimation procedures are available.
By expanding the densities of multivariate generalized gamma convolutions into a tensorized Laguerre basis,
  we bridge the gap and provide performant estimation procedures for both the univariate and multivariate cases.
We provide some insights about performance of these procedures,
  and a convergent series for the density of multivariate gamma convolutions,
  which is shown to be more stable than Moschopoulos's and Mathai's univariate series.
We furthermore discuss some examples.
\end{abstract}




\tableofcontents

\end{frontmatter}
\section{Introduction}

Olof Thorin introduced the Generalized Gamma Convolutions (GGC) class of distributions in 1977
  as a tool to study the infinite divisibility of Pareto~\cite{thorin1977}
  and log-Normal~\cite{thorin1977a} distributions.
The quite natural and pragmatic question of identifying if some distributions are infinitely divisible or not 
  turned out to be theoretically fruitful, 
  and notably gave rise to the concepts of hyperbolic completely monotony and generalized gamma convolutions.
Later,
  Lennart Bondesson extended these concepts in a book~\cite{bondesson1992} 
  which is still today a good reference on the subject.
An introduction to more recent literature about GGC can be found in the two surveys~\cite{james2008,bondesson2015},
  on the probabilistic side of the problem.

The GGC class contains a lot of known and commonly used distributions,
  including heavy-tailed distributions such as Pareto,
  log-Normal and $\alpha$-stable distributions.
This versatility makes the class appealing for the statistical practitioners
  working in various fields such as life and non-life insurance,
  reinsurance,
  anomaly or fraud detection,
  floods analysis and meteorology,
  etc.
Other classes of general approximating distributions with high-order features can be considered.
Mixtures of Erlangs~\cite{cossette2016,lee2012,tijms2003},
  for example,
  have the property of being dense in $L_2$~\cite{tijms2003}
  and allow for fast $k$-MLE estimation algorithms~\cite{nielsen2012}.
On the other hand, 
  they lack crucial closure properties and interpretability that the generalized gamma convolutions have.

Surprisingly,
  on the statistical side,
  very little work can be found about estimation of generalized gamma convolutions.
Only one projection procedure, 
  that is a procedure that projects generalized gamma convolutions onto convolutions of a finite number of gammas,
  was published recently in~\cite{miles2019,furman2017},
Although the resulting convergence result is stunning,
  this procedure cannot handle cases where the incoming density is not already inside the class.
This typically happens for any empirical distribution, 
  due to sampling noise, 
  even if the sampling density is in the class.
Furthermore,
  we show that the theoretical background of the estimation procedure in~\cite{miles2019,furman2017} is inherently univariate,
  and no direct extension to a multivariate case is possible.

The multivariate analogue is still an active research field,
  and, 
  as far as we know,
  has never been considered as a statistical tool in the literature.
In applications, however, statisticians might deal with marginals that are in this class.
A way to handle dependence structure in this interpretable framework seems therefore
  appealing as it would allow studying functions of the random vector,
  like the sum of components,
  under dependence assumptions.
Therefore,
  an estimation procedure for multivariate generalized gamma convolutions that
  takes into account the dependence structure could be useful.
The probabilistic background is already developed:
  based on an original idea from~\cite{cherian1941},
  Bondesson introduces a bivariate extension of the Thorin class,
  and gives some properties about these distributions.
Later,
  \cite{bondesson2009} extends this concept to the multivariate case.
More recently,
  \cite{perez-abreu2012,perez-abreu2014} propose another equivalent framework allowing extension on any cone,
  with application to matrix gamma convolutions.
On the other hand,
  no estimation procedure nor explicit density exist.
We discuss this multivariate class and show that the produced dependence structures are highly flexible:
  non-exchangeability, left tail and right tail dependency are easily achievable.

Regarding tools,
  the literature gives two series expansions for the density of univariate gamma convolutions:
  Mathai~\cite{mathai1982} proposes a density based on Kummer's confluent geometric functions and
  Moschopoulos~\cite{moschopoulos1985} refined it as a convergent gamma series,
  with an explicit truncation error.
Both these densities have the problem of being based on the smallest scale parameter,
  and are not well conditioned when the smallest scale is too small.
Unfortunately,
 this is typically the case when approximating a log-Normal or a Pareto distribution by a finite convolution of gammas.
Therefore,
  no stable procedure for the density computation is available for the entirety of the parameter range.
We are not aware of any density estimation available in the literature for the multivariate case.

We consider here the problem of estimation from samples of multivariate generalized gamma convolutions.
Indeed, 
  in some practical cases such as insurance losses modeling,
  the random variable we model is supposed to be the sum of a random number of independent and identically distributed (i.i.d.) random amounts,
  and hence is inherently divisible.
This is one of the reason that led Thorin to study these concepts:
  he wanted to ensure that the approximating models (a log-Normal or Pareto distribution for example) fulfil this property.
We want to leverage this divisibility for joint statistical analysis.
The parametric divisibility,
  i.e., the possibility of constructing the parametric distribution of the piece easily,
  is therefore an appealing property for a model.
The multivariate generalized gamma convolution class is one of these models
  that allows easy parametric division of the estimated distributions in a multivariate setting,
while staying general enough and allowing for a variety of dependence structure and marginal behaviors.

By the introduction of a specific Laguerre basis,
  which was already considered for non-parametric deconvolution problems by~\cite{dussap2020,mabon2017,benhaddou2019},
  we manage to obtain a series expansion for the density.
Classical deconvolution problems usually consider only one source of signal and one source of noise~\cite{meister2009,mendel1990,zhang2016,sprott1983}.
Although we have here a finite number of signals to be estimated,
  the Laguerre expansion is still a useful tool.
We show that this expansion is quite natural for our problem,
  as it includes and generalizes Moschopoulos's univariate series.
Moreover,
  we introduce a new quantification of the “well-behavior” of a multivariate gamma convolution, 
  that we show to be equivalent to the exponential decay of Laguerre coefficients.
Using this quantification,
  we bridge the gaps and provide a new stable algorithm for the evaluation of theoretical densities
  in both the univariate and multivariate cases,
  as well as consistent parameter estimation procedures which handle both clean data 
  (the density to be estimated is given as a formal function)
  and dirty data (such as an empirical dataset in 64 bits precision). 
The resulting algorithms are implemented in the Julia package \texttt{ThorinDistributions.jl}, 
  available on GitHub\footnote{\texttt{ThorinDistributions.jl}: \url{https://github.com/lrnv/ThorinDistributions.jl}.}
  under a MIT License and archived in Zenodo~\citep{ThorinDistributions}.

After fixing notations,
  Section~\ref{sec:background_and_notations} covers some definitions,
  properties, and algorithms to set up the stage.
Section~\ref{sec:estimation} considers the Laguerre expansion of densities in the multivariate Thorin class,
  studies their regularity and discusses the control of their error through the new quantification
  and develops an estimation procedure for these distributions.
In Section~\ref{sec:investigation},
  we investigate the numerical results of our estimation procedure on several examples.
Section~\ref{sec:conclusion} concludes and gives leads for further work.

\section{Gamma convolutions classes}\label{sec:background_and_notations}

As we deal with a lot of multivariate objects,
  we start by fixing some notations.

We use bold letters such as $\bm a$ to designate generally indexable objects,
  e.g., vectors or matrices,
  and corresponding indexed and possibly unbolded letters,
  such as $a_i$,
  designate values in these objects.
We use Cartesian indexing.
For example,
  if we consider a row-major matrix $\bm r$,
  then $\bm r_i$ denotes the $i^\text{th}$ row of that matrix,
  $r_{i,j}$ the $j^{\text{th}}$ value of that row and if $\bm k = (i,j)$,
  then $r_{\bm k} = r_{i,j}$.

We denote $\lvert\bm x \rvert = x_1 + \ldots + x_d$ the sum of components of a vector $\bm x$.
The product of factorials of component of an integer vector $\bm k$ is denoted $\bm k ! = k_1!...k_d!$,
  and we set $\bm x^{\bm k} = x_1^{k_1}\ldots x_d^{k_d}$.
We denote derivatives of a multivariate function $f$ by 
  $f^{(\bm k)}(\bm x) = \frac{\partial}{\partial^{k_1} x_1}\ldots\frac{\partial}{\partial^{k_d}x_d}(\bm x)$
  and scalar products of vectors by $\langle\bm x,\bm y\rangle = x_1y_1 + \ldots + x_dy_d$.
Finally, we denote $\binom{\bm x}{\bm y} = \prod_i \binom{x_i}{y_i}$ the product of binomial coefficients.

Inequalities,
  sums,
  and products between vectors of same shape are always intended componentwise,
  and standard broadcasting between objects of different shapes applies.
For example,
  $1 + 2\bm x$, $\bm 1 + \bm 2 \bm x$ and $(1 + 2x)_{x \in \bm x}$ all denote the same object,
  which has same shape as $\bm x$.

All over the paper, 
  $d \in \mathbb N_*$ and $\bm X = \left(X_1,...,X_d\right)$ 
  will denote a $d$-variate positive and real-valued random vector,
  abbreviated $X = X_1$ when $d=1$. 

To characterize the distribution of the random vector $\bm X$, 
  we might use several functions of the slack variables $\bm x \in \mathbb R_+^d$ and $\bm t \in \mathbb C^d$,
  with the same convention that $x = x_1$, $t = t_1$ if $d=1$. 
For a (real-valued) random vector $\bm X$ of dimension $d \ge 1$,
  the cumulative distribution function (cdf) of $\bm X$ is $F(\bm x) = \mathbb P\left(\bm X \le  \bm x\right)$.
When they exist,
  the probability density function (pdf) of $\bm X$,
  $f(\bm x) = \frac{\partial }{\partial \bm x}F(\bm x)$,
  its moment generating function (mgf) 
  $M(\bm t) = \mathbb E\left(e^{\langle\bm t, \bm x\rangle}\right)$ 
  or its cumulant generating function (cgf) 
  $K(\bm t) = \ln(M(\bm t))$ can also be used to characterize the distribution.
Note that moment generating functions and cumulant generating functions are not always defined on the whole $\mathbb C^d$.
We nevertheless consider by default $\bm t \in \mathbb C^d$ as there is usually no ambiguity on their domains.  

We denote by $L_2(\mathbb{R}_+^{d},w)$ the set of functions
  that are Lebesgue square-integrable on the space $\mathbb R_{+}^d$
  w.r.t. a weight function $w$,
  and we abbreviate it by $L_2(\mathbb{R}_+^d)$ when $w$ is the identity function.

The following subsections set the stage of our analysis by reviewing common tools from the literature.

\subsection{Definitions and first properties}

A random variable $X$ is said to be Gamma distributed with shape~$\alpha\in \mathbb{R}_{+}^{*}$,
  scale~$s \in \mathbb{R}_{+}^{*}$ and rate~$\frac{1}{s}$,
  which we denote $X \sim \mathcal{G}_{1,1}(\alpha,s)$,
  if it admits the cumulant generating function (cgf):
  $$K(t) = \ln \mathbb{E}\left(e^{tX}\right) =  -\alpha \ln \left(1 - s t\right).$$ 

We denote by $\mathcal{G}_{1,1}$ this class of distributions.
Recall that the cumulant generating function of the sum of independent random variables is the sum of their cumulant generating functions.
Therefore, gamma distributions are clearly infinitely divisible with gamma distributed pieces 
  (same scale and smaller shape).
We now consider independent convolutions of these distributions:

\begin{definition}[$\mathcal{G}_{1,n}$ and $\mathcal{G}_{1}$]\label{univ_ggc_def} Let $X$ be a random variable with support $\mathbb R_+$.
	\begin{enumerate}[label=(\roman*)] 
		\item $X$ is a gamma convolution with shapes $\bm \alpha\in \mathbb{R}_{+}^{n}$ and scales $\bm s\in \mathbb{R}_{+}^{n}$,
		  denoted $X \sim \mathcal{G}_{1,n}(\bm \alpha,\bm s)$,
		  if it has cgf:
		  $$K(t) =  - \sum_{i=1}^n \alpha_i \ln \left(1 - s_i t\right).$$
    We denote this class of distributions $\mathcal{G}_{1,n}$.
		\item $X$ is an (untranslated) generalized gamma convolutions with Thorin measure $\nu$,
		  denoted $X \sim \mathcal{G}_1(\nu)$,
		  if its cgf writes:
		  $$K(t) = -\int\limits_{0}^{\infty}\ln\left(1-st\right) \,\nu(d s)\text{,}$$ where $\nu$ is such that:
		  $$\int\limits_{[1,+\infty)} |\ln(s)| \,\nu(d s) < \infty \text{ and } \int\limits_{(0,1)} s  \,\nu(d s) < \infty\text{.}$$
    We denote by $\mathcal{G}_{1}$ this class of distributions.
	\end{enumerate}
    In both cases, $K$ is defined on its (complex) domain of convergence. 
\end{definition}

The generalized gamma convolutions class $\mathcal G_1$,
  which is commonly called the one dimensional Thorin class,
  has received a lot of interest in the literature.
These limits of convolutions are not dense in the set of positive continuous random variables,
  but they contain many interesting distributions,
  including the  log-Normal and Pareto distributions as historically important cases~\cite{thorin1977,thorin1977a}.
They have several interesting properties:
  noteworthy is the fact that the $\mathcal G_1$ class is closed with respect to independent sums of random variables,
  by construction,
  but also,
  as a more recent result shows,
  by independent products of random variables~\cite{bondesson2015}.
In fact,
  $\mathcal{G}_1$ can be defined as the smallest class that is closed through convolution and that contains gamma distributions.
See~\cite{bondesson1992} for an extensive study of these distributions\footnote{Note that in~\cite{bondesson1992},
  distributions were parametrized through rates $\beta = \frac{1}{s}$ instead.
We chose here to parametrize by scales,
  as it simplifies some of our notations.}.

Remark that if the Thorin measure $\nu$ of $X \sim \mathcal G_1(\nu)$ has a finite discrete support of cardinal $n$,
  then $X \sim \mathcal{G}_{1,n}(\bm \alpha,\bm s  )$ for $\bm\alpha,\bm s$ 
  vectors such that $\nu = \sum\limits_{i} \alpha_i\delta_{s_i}$
  (where $\delta_x$ is the Dirac measure in $x$).

Mathai~\cite{mathai1982} provides a series expression for the density of $\mathcal{G}_{1,n}$
  random variables based on Kummer's confluent hypergeometric functions.
Later,
  Moschopoulos~\cite{moschopoulos1985} refined the result by providing the following gamma series:

\begin{property}[Moschopoulos expansion]\label{prop:mosho_serie} For $X \sim \mathcal{G}_{1,n}(\bm \alpha,\bm s  )$,
  denoting without loss of generality $s_1 = \min \bm s$,
  the density $f$ of $X$ is given by the following uniformly convergent series:
  $$f(x) = \sum_{k=0}^\infty \delta_k f_{\mathcal G_{1,1}\left(\lvert \bm \alpha \rvert+k,\frac{1}{s_1}\right)}(x),\; \forall x \in \mathbb R_+,$$
  where coefficients $\delta_k$ are given in~\cite{moschopoulos1985}.
\end{property}

The dependence of this expansion onto the smallest scale parameter $s_1$ has a major drawback.
When the smallest scale is close to zero,
  the corresponding truncated series is quite unstable.
A simple show-case is the simulation of random numbers from a 
  $\mathcal{G}_{1,2}\left(\left(10,10^{-3}\right),\left(1,10^{-3}\right)\right)$ distribution (which has, e.g., mean $10.000001$ and variance $10.000000001$):
  the Moschopoulos implementation from the R package \texttt{coga}~\cite{hu2020}
  gives a density that evaluates to $0$ on all random numbers it itself simulated,
  which is obviously wrong.
Mathai's method,
  implemented in the same package,
  produces the same result.
Sadly,
  as later implementation will show,
  parameters that correspond to approximations of useful distributions such as  log-Normal,
  Weibull or Pareto usually have very small scales and trigger the same instability.

A likelihood approach to fit distributions in the Thorin class is therefore not practical, 
  as there is no stable density,
  which is part of the reason for which there are currently no estimation procedures in the literature.

In~\cite{bondesson1992,bondesson2009},
  Bondesson defined a class of multivariate convolutions of gamma distributions based on the following idea from Cherian~\cite{cherian1941}.
Suppose that $Z_0,Z_1$ and $Z_2$ are three independent gamma random variables
  with (respective) shapes $\alpha_i$ and scales $s_i$,
  with w.l.o.g. $s_0 = 1$.
Then the random vector $\left(s_1Z_0 + Z_1,s_2Z_0 + Z_2\right)$ has a meaningful structure of dependence,
  while retaining gamma marginals.
Indeed, its cgf writes:
  $$K(\bm t) = -\alpha_1\ln\left(1-s_1t_1\right) -\alpha_2\ln\left(1-s_2t_2\right) -\alpha_0\ln\left(1-s_1t_1 -s_2t_2\right).$$

This construction can be extended to what Bondesson called multivariate generalized gamma convolutions:

\begin{definition}[$\mathcal{G}_{d,n}$ and $\mathcal{G}_{d}$]\label{def:Gd} 
  Let $\bm X$ be a $d$-variate random vector with support $\mathbb R_+^d$.
	\begin{enumerate}[label=(\roman*)] 
		\item $\bm X$ is a multivariate gamma convolution with shapes $\bm \alpha \in \mathbb R_+^n$ 
      and scale matrix $\bm s = \left(s_{i,j}\right) \in \mathcal{M}_{n,d}\left(\mathbb R_{+}\right)$,
		  denoted by $\bm X \sim \mathcal{G}_{d,n}(\bm \alpha,\bm s  )$ if it has cumulant generating function:
      $$K(\bm t) = \sum\limits_{i=1}^n -\alpha_i \ln\left(1 - \scalarprod{\bm s_i}{\bm t}\right).$$
    We denote by $\mathcal{G}_{d,n}$ this class of distributions.
		\item $\bm X$ is an (untranslated) multivariate generalized gamma convolutions with Thorin measure $\nu$,
		  denoted $\bm X \sim \mathcal G_{d}(\nu)$,
		  if it is a weak limit of $d$-variates gamma convolutions.
		Its cgf writes:
      $$K(\bm t) =  - \int \ln\left(1 - \langle\bm s,\bm t\rangle \right) \nu(d\bm s),$$ for a positive measure $\nu$ of $\mathbb R_{+}^d$,
		  the Thorin measure,
		  with suitable integration conditions (see~\cite{perez-abreu2012}).
		We denote by $\mathcal{G}_{d}$ this class of distributions.
    In both cases, $K$ is defined on its (complex) domain of convergence. 
	\end{enumerate}
\end{definition}

Note that $\bm X \sim \mathcal{G}_{d,n}(\bm \alpha,\bm s  )$
  if and only if there exists a vector $\bm Z$ of $n$ independent unit scale gamma random variables
  with shapes $\alpha_1,...,\alpha_n$ such that $\bm X = \bm s' \bm Z$.
The $j^{\text{th}}$ marginal $X_j$ has distribution $\mathcal{G}_{1,n}(\bm \alpha,\bm s_{.,j})$.
We can interpret this as a decomposition of the random vector $\bm X$ into 
  an additive risk factor model with gamma distributed factors.
Note that we allow scales to be zero,
  and some factors might therefore appear only in some marginals.
As in the univariate case,
  we close the class by taking convolutional limits.
For the analysis of these distributions,
  we refer to~\cite{bondesson1992,bondesson2009},
  but also~\cite{perez-abreu2012} which uses a slightly different but equivalent parametrization,
  allowing a generalization to other cones than $\mathbb R_{+}^{d}$.

Consider a distribution $F$, 
  not necessarily in $\mathcal G_d$,
  that generates our observations.
We are here interested in parametric estimators in $\mathcal G_d$ of this distribution,
  with particular interest for estimators in $\mathcal{G}_{d,n} \subsetneq \mathcal G_d$.
Indeed,
  $\mathcal{G}_{d,n}$ models have a structure that allows fast simulation,
  and that provides meaningful insight about the dependence structure,
  since a $\mathcal{G}_{d,n}$ distribution follows an additive risk factor model.
On the other hand,
  distributions in $\mathcal{G}_d$ with a non-atomic Thorin measure are hard to sample
  (see~\cite{bondesson1982,ridout2009} for potential solutions to this problem).
They have no known closed-form density or distribution function,
  and even the cumulant generating function requires integration to be evaluated.

Before diving into statistical considerations and estimation of these distributions in Section~\ref{sec:estimation},
  we give here some properties about the dependence structure that can be achieved in the $\mathcal{G}_d$ class,
  and about the regularity of the density.

Recall that a real valued random vector $\bm X$ with distribution function $F$ and 
  marginal distribution functions $F_1,...,F_d$ is said to be independent 
  if $F(\bm x) = \prod\limits_{i=1}^d F_i(x_i)$ and comonotonic
  if $F(\bm x) = \min \left(F_1(x_1),...,F_d(x_d)\right)$. 
For more details on dependence structures,
  see~\cite{nelsen2007}, 
  in particular Theorems 2.3.3 (Sklar), 2.5.4 and 2.5.5 (Fréchet-Hoeffding bounds).
Note that $\mathcal G_{d,1}$ random vectors are, by construction, comonotonic.

Between these two extreme cases,
  there are many available structures of dependence for a $\mathcal G_d$ random vector. 
The following properties give more insights about the relation between
  the dependence structure of the random vector
  and the dependence structure of its Thorin measure:

\begin{property}[The support of the Thorin measure]\label{prop:special_thorin_measures} Let $\bm X \sim \mathcal{G}_d(\nu)$,
  and let $Sup(\nu)$ be the support of $\nu$.
Then we have:
	
	\begin{enumerate}[label=(\roman*)]
		\item The marginals of $\bm X$ are mutually independent if and only if
      $$Sup(\nu) \subset S_{\indep} 
        = \BigCup\limits_{j=1}^d \left(\Times_{i=1}^{j-1} \{0\}\right)\times \mathbb R_{+}^{*} \times \left(\Times_{i=j+1}^{d} \{0\}\right).$$
    In this case, 
      if we denote for all $j \in 1,...,d$, 
      $\Omega_j = \{\bm s \in \mathbb R_+^d, \forall k \in \{1,...,d\}\setminus\{j\},\, s_k = 0\}$ the different parts of the support, 
		  $\nu(\bm s) = \sum\limits_{j=1}^{d} \nu_j(s_j) \mathbb{1}_{\bm s \in \Omega_j}$
      where $\nu_j$ denotes the Thorin measure of the $j^{\text{th}}$ marginal,
		  and the total masses of $\nu,\nu_1,...,\nu_d$ are such that
      $\nu(\mathbb R_{+}^d) = \nu(S_{\indep}) = \sum_{i=1}^d \nu_i(\mathbb R_{+})$.
		\item The marginals of $\bm X$ are comonotonic if and only if there exists a constant
      $\bm c\in \mathbb{R}_{+}^d,\lVert \bm c\rVert = 1$ such that
      $$Sup(\nu) \subset S_{\bm c} = \left\{\bm r \in \mathbb R_{+}^d,\; \frac{\bm r}{\lVert \bm r \rVert} = \bm c\right\}.$$
    In this case,
		  the multivariate Thorin measure and its marginals have all the same total mass.
		\item The r.v. $\bm X$ has a Lebesgue-square-integrable density if $D \ge d$,
		  where $D$ denotes the minimum integer such that there exists constants
      $\bm c_1,...,\bm c_D$ such that $Sup(\nu) \subset \BigCup\limits_{i=1}^d S_{\bm c_i}$.
		On the other hand if $D < d$, $\bm X$ is a singular random vector with a $D$-dimensional support
      (which is obviously the case when $\nu$ is atomic with less than $d$ atoms).
	\end{enumerate}
\begin{proof}
  See Appendix~\ref{sec:proofs}.
\end{proof}
\end{property}

The link given at point $(iii)$ of Property~\ref{prop:special_thorin_measures} 
  between the regularity of the density and the constant $D$ (which, for atomic Thorin measures, is the rank of the matrix $\bm s$), 
  will be investigated and quantified further in the next section. 

Between the independent and the comonotonic cases,
  there is a wide range of achievable dependence structures.
For example,
  since the class $\mathcal{G}_d$ is closed w.r.t. convex convolution,
  every (positive) value of dependence measures,
  such as Kendall tau or Spearman rho,
  are achievable.
Some examples in Section~\ref{sec:investigation} also exhibit highly asymmetrical dependence structure and tail dependency.

For the study of the relation between the dependence structure of $\bm X$ and the dependence structure of $\nu$,
  the distribution proposed in Example~\ref{prop:curious_dist} is of interest.

\begin{example}[A curious distribution]\label{prop:curious_dist}
  Let $\nu$ be a one-dimensional measure defined through $\,\nu(d x) = \mathbb 1_{x \in [0,1]} dx$.
	Then $\mathcal{G}_{1}(\nu)$ has cumulant generating function:
	$$K(t) = 1 - \frac{(t-1)}{t}\ln(1-t)$$
	
	Moreover,
	  if for all $n\in \mathbb N,G_{j,n} \sim \mathcal G_{1,1}\left(\frac{1}{n},\frac{j}{n+1}\right)$ are independent,
	  then: $$\sum_{j=1}^{n} G_{j,n} \xrightarrow[n \to \infty]{} \mathcal{G}_{1}(\nu) \quad\text{(in distribution).}$$
    \begin{proof}
      See Appendix~\ref{sec:proofs}.
    \end{proof}
\end{example} 

This random variable has the nice property that its Thorin measure is uniform on [0,1].
However,
  its use is limited by the following negative result:

\begin{remark}[The no-bijection result] Let $\bm X \sim \mathcal{G}_{d}(\nu)$ have marginal cdf $F_1,...,F_d$,
  and denote by $G$ the cdf associated to the distribution of Example \ref{prop:curious_dist}.
	Then,
	  assuming that $\nu$ is absolutely continuous,
	  the random vector 
	  $$\left(G^{-1}\left(F_i(X_i)\right)\right)_{i \in 1,...,d}$$
	  has uniform marginal Thorin measures,
	  but is not $\mathcal{G}_{d}(c_\nu)$-distributed,
	  where $c_{\nu}$ is the copula of $\nu$.
\end{remark}

There is in fact a bijection between the copula of the random vector and the copula of its Thorin measure,
  but conditionally on the marginal distributions.
Therefore,
  an estimation scheme that separates the dependence structure from the marginals seems not possible.

Before discussing the estimation of these distributions,
  we still need to introduce in the next subsection some specific integrals,
  which correspond to derivatives of the moment generating function and the cumulant generating function of a random vector.

\subsection{Shifted moments and shifted cumulants}

We formally introduce the shifted moments and the cumulants of a random vector,
  and present a known and useful result that maps moments to cumulants and vice-versa.

Recall that $\bm X$ is a random vector of dimension $d$,
  with moment generating function $M$ and cumulant generating function $K$ defined by:
  $$\forall \bm t \in \mathbb C^{d},\; M(\bm t) = \mathbb E\left(e^{\langle\bm t,\bm X\rangle}\right) \text{ and } K(\bm t) = \ln M(\bm t)\text{.}$$

We now introduce specific notations for the derivatives of these two functions.

\begin{definition}[Shifted moments and cumulants] For $\bm X$ a random vector of dimension $d$,
	  for $\bm i \in \mathbb N^d$,
	  define the $\bm t$-shifted $\bm i^{\text{th}}$ moments $\mu_{\bm i,\bm t}$ and cumulants $\kappa_{\bm i,\bm t}$ as the $\bm i^{\text{th}}$ derivatives of (respectively) $M$ and $K$,
	  taken at $\bm t$:
    $$\mu_{\bm i,\bm t} = M^{(\bm i)}(\bm t) \text{ and } \kappa_{\bm i,\bm t} = K^{(\bm i)}(\bm t)\text{.}$$
\end{definition}

They correspond to the moment and cumulants of an exponentially tilted version of the random vector $\bm X$,
  sometimes called the Escher transform~\cite{escher1932} of $\bm X$.
Although standard Monte-Carlo estimators for shifted moments are unbiased and easy to compute,
  shifted cumulants are a little harder to estimate from i.i.d. samples of a random vector.
We refer to~\cite{smith2020,smith2020a} for some unbiased estimators of multivariate cumulants,
  and~\cite{krutto2019} for an application to the estimation of stable laws.
To switch from moments to cumulants,
  Property~\ref{bijection_mu_kappa} gives a bijection between the two:

\begin{property}[Bijection between moments and cumulants]\label{bijection_mu_kappa}
	For a given shift $\bm t \in \mathbb C^{d}$ and $\bm m \in \mathbb{N}^d$,
	  the sets $\left(\mu_{\bm i,\bm t}\right)_{\bm i \le \bm m}$ and $\left(\kappa_{\bm i,\bm t}\right)_{\bm i \le \bm m}$
    are in bijection.

	\begin{proof}
    Since $\left(\mu_{\bm i,\bm t}\right)_{\bm i \le \bm m}$ and
    $\left(\kappa_{\bm i,\bm t}\right)_{\bm i \le \bm m}$ are respective Taylor coefficients 
    at $\bm t$ of $\exp \circ K$ and $K$, 
    they are in bijection through the multivariate Faà Di Bruno formula~\cite{constantine1996,mishkov2000,nardo2011}.
	\end{proof}

\end{property}

Remark that this bijection can be expressed in many ways:
  most analysis in the literature uses multivariate Bell Polynomials~\cite{bernardini2005,cvijovic2011,kim2019,melman2018} for this task.
For computational purposes,
  there exists some practical functions that give coefficients of Bell polynomials~\cite{withers1994,hardy2006},
  even for $d>1$.
However,
  in all generality,
  the number of coefficients to compute and store is exponential with $\bm m$ making this method quite unpractical,
  as the Bell polynomials are combinatorial in nature.
To switch from cumulants to moments,
  the good property of the exponential function of being its own derivative allows a recursive approach,
  which is described in full details in~\cite{miatto2019}. 
We adapted the notation and simplified the indexes conventions from~\cite{miatto2019}, 
  which allowed,
  after adding a slack variable to handle the dimensionality index,
  for an even faster implementation.
The expression in Algorithm~\ref{algo:Miatto2019} below uses derivatives of the cgf and mgf as it is the case that we need,
  but the derivative of the exponential of any function could be computed through exactly the same algorithm,
  knowing the derivatives of the exponent.

\begin{algorithm}[H]\label{algo:Miatto2019}
	\SetAlgoLined
	\KwIn{$\bm m \in \mathbb{N}^d$, Shifted Cumulants $\left(\kappa_{\bm k,\bm t}\right)_{\bm k \le \bm m}$}
	\KwResult{Shifted Moments $\left(\mu_{\bm k,\bm t}\right)_{\bm k \le \bm m}$}
	Set $\mu_{\bm 0,\bm t} = \exp\left(\kappa_{\bm 0,\bm t}\right)$\\
	Set $\mu_{\bm k,\bm t} = 0$ for all $\bm k \neq \bm 0$\\
	\ForEach{$\bm k:\;\bm 0 \neq \bm k \le \bm m$}{ 
		Set $d$ as the index of the first $k_i$ that is non-zero.\\
		Set $\bm p = \bm k$\\
		Set $p_d = p_d - 1$\\
		\ForEach{$\bm l:\; \bm l \le \bm p$}{
			$\mu_{\bm k,\bm t} \pluseq \left(\mu_{\bm l,\bm t}\right)  \left(\kappa_{\bm k - \bm l,\bm t}\right)  \binom{\bm p}{\bm l}$
		}
	}
	Return $\bm \mu$
	\caption{Recursive computation of $\bm \mu$ from $\bm \kappa$. (see~\cite{miatto2019})}
\end{algorithm}

Note that the main loop of Algorithm \ref{algo:Miatto2019} must be done in the right order such that each $\mu_{\bm k}$ is computed before being used.
This recursive formulation has stunning results and is at the moment the fastest way this computation can be done%
  \footnote{The complexity is not exponentially increasing in the size of the arrays,
  on the contrary to naive implementations of Faà Di Bruno's formula,
  see~\cite{miatto2019}.}.
Unfortunately,
  no equivalently fast procedure can be found the other way around,
  as the derivatives of the log function are more complicated.

We now dive into the estimation of multivariate gamma convolutions.
We highlight in the next subsection the univariate projection procedure from~\cite{furman2017,miles2019}.

\subsection{Projection from \texorpdfstring{$\mathcal{G}_{1}$}{G1} to \texorpdfstring{$\mathcal{G}_{1,n}$}{G1n}}\label{sec:projection_furman}

Suppose that we have a density in $\mathcal{G}_1$, given as a formal function.
Then,
  for a certain $n \in \mathbb{N}$,
  the procedure provided by Miles, Furman \& Kuznetsov in~\cite{miles2019,furman2017},
  hereafter called “the MFK procedure”,
  gives expression for the shapes $\bm \alpha$ and scales $\bm s$ of a $\mathcal{G}_{1,n}(\bm \alpha,\bm s)$
  that is exponentially close to the desired density.
We will not describe the algorithm here, 
  we refer to the original article for the details as the exposition is quite long,
  but we will comment on the results and caveats and use it as a comparison basis for later experiments.

The MFK procedure solves a univariate moment problem for the Thorin measure. 
This moment problem can be highlighted by deriving the cumulant generating function of a $\mathcal G_1$ random vector.
Indeed,
  remark that, for $X \sim \mathcal G_1(\nu)$ and $t\ge 0$,
  $$\kappa_{k,-t} = \left(k-1\right)!\int \frac{s^k}{\left(1 + st\right)^k}\,\nu(d s),$$ 
  and use the change of variables $x = \frac{s}{1+st}$,
  mapping $\mathbb{R}_{+}$ to $\left[0,\frac{1}{t}\right]$.
You obtain a moment problem of the form:
  $$\xi_k = \int\limits_{\left[0,\frac{1}{t}\right]} x^k\, \xi(d x),\; \forall k \le m,$$
  where $\bm \xi = \left(\xi_k\right)_{k \in \mathbb N} = \left(\frac{\kappa_{k,-t}}{\left(k-1\right)!}\right)_{k \in \mathbb N}$ 
  are moments that the measure $\xi$ needs to fit.

We can then obtain a solution with $n$ atoms for the measures $\xi$ and $\nu$ by solving a truncated version of this moment problem.
In the univariate case, 
  $K^{(1)}$ is a one-dimensional Stieltjes function,
  and therefore this moment problem can be efficiently solved via Padé approximants (see~\cite{derevyagin2007}).

The first drawback of MFK is that 
  the computed moments $\bm \xi$ need to be inside the moment cone of positive measure on $\left[0,\frac{1}{t}\right]$.
For the projection of a formal density already in $\mathcal G_1$ on $\mathcal G_{1,n}$,
  this condition imposes estimations of shifted moments (from which $\bm \xi$ are derived)
  with at least $1024$ bits of precision (about $300$ digits). 
On the other hand,
  if $\bm \xi$ is \emph{not} a sequence of moments of some measure on $\left[0,\frac{1}{t}\right]$,
  then the moment problem becomes unsolvable,
  and 
  MFK fails to provide correct (within boundaries) atoms and weights for the measures $\xi$ and $\nu$.
From an empirical dataset,
  even if the true distribution lies in $\mathcal{G}_1$,
  evaluation of $\bm \mu$ through Monte-Carlo with enough precision is therefore impossible.

Through this moment problem interpretation,
  the approach could nevertheless be carried out in higher dimensions,
  for projections of densities from $\mathcal G_d$ to $\mathcal G_{d,n}$,
  with an even stronger integration precision requirement as the dimension increases.
However,
  the results about the complete factorization of the Padé approximants denominator of a Stieltjes function,
  which drastically simplifies the solution of the moment problem,
  are not true for $d > 1$.
This possibility is nevertheless discussed in the next section.

\section{Construction of an estimator in \texorpdfstring{$\mathcal{G}_{d,n}$}{Gdn}}\label{sec:estimation}

We attempt a multivariate extension of the MFK algorithm.
In a multivariate setting,
  through the moment problem interpretation of MFK algorithm developed in Subsection~\ref{sec:projection_furman},
  we can express the corresponding multivariate moment problem.

For $\bm X \sim \mathcal{G}_{d}\left(\nu\right)$,
  consider $K$ the cumulant generating function of $\bm X$,
  given in Definition \ref{def:Gd} as:
  $$K(\bm t) = - \int \ln\left(1 - \scalarprod{\bm s}{\bm t}\right) \,\nu(d \bm s),$$
where $\bm t$ is in the (complex) domain of convergence of $K$.

Assuming $\lvert \bm i \rvert \ge 1$ and the real parts $\Re\left(\bm t\right)$ of $\bm t$ are all non-positives,
  we have that
  $$\frac{\partial^{\bm i}}{\partial^{\bm i} \bm t} \ln \left(1 - \langle\bm s,\bm t\rangle \right) 
    = - \bm s^{\bm i} \left(\lvert \bm i \rvert -1\right)! \left(1 - \langle\bm s,\bm t\rangle \right)^{- \lvert \bm i \rvert},$$
  and, therefore,
  \begin{align}
  \kappa_{\bm i,\bm t} = K^{(\bm i)}(\bm t) 
  &= - \frac{\partial^{\bm i}}{\partial^{\bm i} \bm t} \int\limits \ln \left(1 - \langle\bm s,\bm t
    \rangle\right) \,\nu(d \bm s) \nonumber\\
  &= \int  \bm s^{\bm i} \left(\lvert \bm i \rvert -1\right)! \left(-1\right)^{\lvert \bm i \rvert-1} \left(1 - \langle\bm s,\bm t
    \rangle\right)^{- \lvert \bm i \rvert} \,\nu(d \bm s) \nonumber\\
  & = \left(\lvert \bm i \rvert -1\right)!\int  \frac{\bm s^{\bm i}}{ \left(1 - \langle\bm s,\bm t
    \rangle\right)^{\lvert \bm i \rvert}} \,\nu(d \bm s).\label{final_eq_kappa}
  \end{align}

Note that $\Re\left(\bm t\right) \le \bm 0$ is inside the region of convergence of the function,
  and use now the (bijective, continuous) change of variables:
  $$\bm x = \frac{\bm s}{1-\langle\bm s,\bm t\rangle} \iff \bm s = \frac{\bm x}{1 + \langle\bm x,\bm t\rangle},$$ 
  which goes from $\mathbb R_{+}^{d}$ to the simplex 
  $\Delta_{d}(-\bm t) = \left\{\bm x \in \mathbb R_{+}^d, \; \,\langle\bm x,-\bm t\rangle \le 1\right\}$,
  yielding an equivalent problem of finding the measure $\xi$ that solves the more standard moment problem:
  \begin{equation}\label{multivariate_moment_problem}
  \forall\,\bm i \le \bm m,\; \xi_{\bm i} = \int\limits_{\Delta_{d}(-\bm t)} \,\bm x^{\bm i}\xi\left(d\bm x\right),
  \end{equation}
  where we denoted 
  $\bm \xi = \left(\frac{\kappa_{\bm i,\bm t}}{(\lvert \bm i \rvert -1)!}\right)_{\bm i \le \bm m}$
  the moments that $\xi$ needs to match according to Equation~\eqref{final_eq_kappa}.

Finding an $n$-atomic measure $\xi$ solution to these equations is equivalent to finding the parametrization
  of a $\mathcal{G}_{d,n}$ distribution that fulfills the cumulant constraints.
The corresponding Generalized Moment Problem is known as a hard problem in the literature,
  but can be solved through semi-definite positive relaxations,
  following, e.g.,
 \cite{helton2012,nie2012,henrion2012a}.
Sadly,
  these algorithms and the corresponding literature focus on the first case,
  where $\bm \xi$ are,
  indeed,
  moments of a measure supported on the simplex $\Delta_{d}(-\bm t)$,
  which is the case only if $\bm X \in \mathcal{G}_{d,n}$ and if $\bm \xi$ are computed with enough precision,
  as it was in univariate settings.

If the knowledge about the distribution of $\bm X$ is given through empirical observations,
  with sampling noise,
  or if the true distribution is $\emph{not}$ in the $\mathcal{G}_{d,n}$ class,
  these moments will not belong to the moment cone,
  and the moment problem will have no solution.
By projecting the moments onto the moment cone,
  we could force the known algorithms to provide a solution,
  but the projection onto this cone is not an easy task (see~\cite{didio2018,didio2018a} for details on moment cones).

The estimation of an $n$-atomic Thorin measure from samples of the random vector $\bm X$ therefore rises several questions,
  both in univariate and multivariate cases.
If we empirically compute cumulants,
  we will stand outside the moment cone and the moment problem will have no solution.
Then,
  should we consider Equation~\eqref{multivariate_moment_problem} as a multi-objective optimization problem ?
Do we seek a Pareto front,
  or can we consider that some moments are less important than others ?
If we do,
  how do we weight the several objectives ?

To produce a meaningful loss to minimize, 
  we will weight these moment conditions.
For this purpose,
  we leverage a certain Laguerre basis of $L_2(\mathbb{R}_+^d)$ to expand the density of random vectors in $\mathcal{G}_{d,n}$,
  and to construct a coherent loss for this,
  potentially impossible to solve,
  moment problem.

\subsection{The tensorized Laguerre basis} 

The standard Laguerre polynomials~\cite{mustapha2010} form an orthogonal basis of the set $L_2(\mathbb{R}_+,e^{-x})$
  of square integrable functions with respect to the weight function $x \mapsto e^{-x}$.
From these polynomials,
  we can extract the following orthonormal basis of the set $L_2(\mathbb{R}_+^d)$ 
  of functions that are Lebesgue square-integrable on $\mathbb R_+^d$:
\begin{definition}[Laguerre function]\cite{comte2015,mabon2017,dussap2020}\label{def:laguerre_basis}
	For all $\bm k \in \mathbb{N}^d$,
  define the Laguerre functions $\phi_{\bm k}$ with support $\mathbb{R}_+^d$ by:
  $$\phi_{\bm k}(\bm x) = \prod\limits_{i=1}^d \phi_{k_i}(x_i) 
    \text{ where for } k \in \mathbb{N},\;
    \phi_k(x) = \sqrt{2}e^{-x}\sum\limits_{\ell=0}^k \binom{k}{\ell} \frac{(-2x)^\ell}{\ell!}.$$
\end{definition}

For every function $f \in L^2(\mathbb{R}_+^d)$,
  for every $\bm k \in \mathbb N^d$,
  coefficients of $f$ in the basis are denoted by:
  $$a_{\bm k}\left(f\right) = \int\limits_{\mathbb R_+^d} f(\bm x) \phi_{\bm k}(\bm x) \,d\bm x,$$
  and we have,
  since the basis is orthonormal,
  that $f(\bm x) = \sum\limits_{\bm k \in \mathbb N^d} a_{\bm k}(f) \phi_{\bm k}(\bm x)$ 
  and that $\lVert f\rVert_{2}^2 = \sum\limits_{\bm k \in \mathbb{N}^d}a_{\bm k}^2(f)$.

Furthermore,
  for $\bm X$ a random vector with density $f \in L_{2}\left(\mathbb R_+^d\right)$ and shifted moments $\bm \mu$,
  we have by simple plug-in of $f$ and $\phi_{\bm k}$ into the expression of $a_{\bm k}$ that:
  \begin{equation}\label{eq:laguerre_coef_expression}
	a_{\bm k}(f) = \sqrt{2}^d \sum\limits_{\bm \ell \le \bm k} \binom{\bm k}{\bm \ell} \frac{(-2)^{\lvert \bm \ell \rvert}}{\bm \ell !} \mu_{\bm \ell,-\bm 1}.
	\end{equation}

We now denote $a_{\bm k}(\bm \alpha,\bm s)$ the $\bm k^{\text{th}}$ Laguerre coefficient of the $\mathcal G_{d,n}(\bm \alpha,\bm s)$ distribution.

\begin{example}[$\bm X \sim \mathcal{G}_{d,1}(\alpha,\bm s)$]\label{ex:d1-inversion} For $\bm X \sim \mathcal G_{d,1}(\alpha,\bm s)$,
  there exists an explicit bijection between the $d+1$ first Laguerre coefficients (with index $\bm k$ s.t. $\lvert\bm k\rvert \le 1$) and the $d+1$ parameters,
  whose expression is given in the proof.
    \begin{proof}
      See Appendix~\ref{sec:proofs}.
    \end{proof}
\end{example}

This example is quite peculiar since by Property \ref{prop:special_thorin_measures} the random vector is singular.
Note that the edge case $\exists i \in 1,..,d,\;r_i =0$,
  in which the $i^{\text{th}}$ marginal is identically zero,
  is included in the previous result.

We do not know if the same kind of inversion could be carried out analytically for the coefficients of $\mathcal{G}_{d,n}$ densities,
  since Laguerre coefficients are not as trivially expressed in all generality.
Furthermore,
  this kind of parametric inversion would not work for empirical datasets or densities outside the class,
  and is therefore not our focus here.
However,
  we can efficiently compute coefficients from parameters of a $\mathcal{G}_{d,n}$ random vector through Algorithm~\ref{algo:DensityOMGC},
  based on Algorithm~\ref{algo:Miatto2019}:

\begin{algorithm}[H]\label{algo:DensityOMGC}
	\SetAlgoLined
	\KwIn{Shapes $\bm \alpha \in \mathbb R_{+}^n$, scales $\bm s \in \mathcal{M}_{n,d}\left(\mathbb R_{+}\right)$, and truncation threshold $\bm m \in \mathbb N^d$}
	\KwResult{Laguerre coefficients $\bm a = \left(a_{\bm k}\right)_{\bm k \le \bm m}$of the $\mathcal G_{d,n}\left(\bm \alpha,\bm s  \right)$ density}
  \KwSideProd{$-\bm 1$-shifted cumulants $\bm \kappa$ and moments $\bm \mu$ op to order $\bm m$ of the $\mathcal G_{d,n}(\bm \alpha, \bm s)$ distribution.}
	Compute the simplex version of the scales $\bm x_i = \frac{\bm s_i}{1 + \lvert \bm s_i \rvert}$ for all $i \in 1,...,n$.\\
	Let $\kappa_{\bm 0} = \sum\limits_{i=1}^{n} \alpha_i \ln\left(1 - \lvert \bm x_i \rvert\right)$\\
	Let $a_{\bm 0} = \mu_{\bm 0} = \exp\left(\kappa_{\bm 0}\right)$\\
	\ForEach{$\bm 0 \neq \bm k \le \bm m$}{ 
		Let $a_{\bm k} = \mu_{\bm k} = 0$\\
		Let $d$ be the index of the first $k_i$ that is non-zero.\\
		Let $\bm p = \bm k$\\
		Set $p_d = p_d - 1$\\
		Let $\kappa_{\bm k} = \left(\lvert \bm k \rvert - 1\right)!\; \sum\limits_{i=1}^n \alpha_i \bm x_i^{\bm k}$ according to Eq.~\eqref{final_eq_kappa}\\
		\ForEach{$\bm l \le \bm p$}{
			Set $\mu_{\bm k} \pluseq \left(\mu_{\bm l}\right)  \left(\kappa_{\bm k - \bm l}\right)  \binom{\bm p}{\bm l}$ according to Algorithm \ref{algo:Miatto2019}\\
			Set $a_{\bm k} \pluseq \mu_{\bm l} \binom{\bm k}{\bm l} \frac{(-2)^{\lvert \bm l \rvert}}{\bm l!}$ according to Eq.~\eqref{eq:laguerre_coef_expression}\\
		}
		Set $a_{\bm k} \pluseq \mu_{\bm k} \frac{(-2)^{\lvert \bm k \rvert}}{\bm k!}$
	}
	$\bm a = \sqrt{2}^d \bm a $\\
	Return $\bm a$
	\caption{Laguerre coefficients of $\mathcal{G}_{d,n}(\bm \alpha,\bm s)$ distributions.}
\end{algorithm}

The complexity of Algorithm~\ref{algo:DensityOMGC} is quadratic in $\lvert\bm m\rvert$.
Note that,
  as in Algorithm~\ref{algo:Miatto2019},
  computations need to be performed in the right order.
Sometimes, 
  the Laguerre coefficients $\bm a = \left(a_{\bm k}\right)_{\bm k \le \bm m}$ overflow the Float64 limits,
  but the implementation we provide in the Julia package \texttt{ThorinDistributions.jl}~\citep{ThorinDistributions},
  ensures that the computations do not overflow by using multiple precision arithmetic when needed.
Furthermore,
  as in Algorithm~\ref{algo:Miatto2019},
  the use of Miatto's fast recursion gives this algorithm a good efficiency,
  even for reasonably large $d,n,\bm m$.

By reordering the coefficients and leveraging generalizations of Laguerre polynomials,
  we have the following link with Moschopoulos's density:

\begin{remark}[Generalized Laguerre basis and Moschopoulos density] 
  Note that Laguerre expanded densities are expressed as (infinite) gamma mixtures, 
    as was Moschopoulos's density in the univariate case (see Property~\ref{prop:mosho_serie}).
  Furthermore, 
    following~\cite{comte2015},
    the Laguerre basis is defined through univariate orthogonal polynomials w.r.t the weight function $e^{-x}$, i.e a $\mathcal G_{1,1}(1,1)$ density.
  If, 
    instead, 
    we use the weight function $x^\rho e^{-x}, \rho \ge 0$, corresponding to a $\mathcal G_{1,1}(1+\rho,1)$ density,
    the associated orthogonal polynomials are the so-called generalized Laguerre polynomials,
    and a slightly different associated orthonormal basis of $L_2(\mathbb R_+^d)$ is obtained.
	In the one-dimensional case,
	  if we chose $\rho$ to be the total mass of the Thorin measure,
	  we retrieve Moschopoulos series from Property~\ref{prop:mosho_serie} as a Laguerre expansion.
\end{remark}

Sadly,
  even if mixtures of gammas are easier to fit by $k$-MLE procedure~\cite{nielsen2012,schwander2013},
  we have no way of identifying the subspace of coefficients that would match a true convolution of gammas:
  as~\cite{bondesson1992} noted in the univariate case,
  generalized gamma convolutions can be expressed as mixtures of gammas,
  but there is no simple reverse mapping.
Last but not least,
  if the expansion through generalized Laguerre polynomials with a parameter corresponding
  to the total mass of the Thorin measure would converge faster,
  we would have no way of estimating this parameter beforehand from data,
  and for many useful cases it would be infinite (log-Normals among others).

We can now compute efficiently these Laguerre coefficients for $\mathcal G_{d,n}$ distributions.
In the next subsection, we provide a control on their decay.

\subsection{Exponential decay of Laguerre coefficients}

We now consider the convergence of these Laguerre expansions. 
Under simple technical assumptions on the parameters of a $\mathcal{G}_{d,n}$ random vector,
  Laguerre coefficients are (uniformly) exponentially decreasing.
Theorem~\ref{prop:bounds} provides this bound:

\begin{theorem}[Exponential decay of Laguerre coefficients]\label{prop:bounds} 
  For any $\epsilon, \epsilon'$ s.t. $\epsilon > \epsilon' > 0$ and any dimension $d$,
    there exists a finite positive constant $B(d,\epsilon')$ such that 
    the Laguerre coefficients $\left(a_{\bm k}\right)_{\bm k \in \mathbb N^d}$ 
    of any $d$-variate $\epsilon$-well-behaved gamma convolution satisfy: 
    $$\lvert a_{\bm k}\rvert  \le B(d,\epsilon')(1+\epsilon')^{-\lvert k \rvert}.$$
\end{theorem}

Of course, we did not define yet the concept of $\epsilon$-well-behaved gamma convolution, which will be central in the proof. 
Before exposing the main proof, 
  we start by a few technical statements.
Lemma~\ref{lem:mobius} gathers some useful properties about a particular univariate Möbius transform.
\begin{lemma}[Property of a Möbius transform]\label{lem:mobius}
  Let $h$ be the Möbius transform: 
    $$h(t) = \frac{t+1}{t-1}.$$
  For $a \in \mathbb C$ and $b \in \mathbb R_+$, we denote the disc with center $a$ and radius $b$ by:
  $$D(a,b) = \left\{t \in \mathbb C, \lvert t-a \rvert < b\right\}.$$
  Also denote $c(t) = \frac{t^2+1}{t^2-1}$ and $r(t) = \left\lvert \frac{2t}{t^2-1}\right\rvert$,
    and finally denote $\mathbb C_-$ and $\mathbb C_+$ the left half and right half of the complex plane. 
  The following holds: 
  \begin{enumerate}[label=(\roman*)]
    \item $h$ is its own inverse, 
      i.e., $h(h(t)) =t$,
      and $h$ has a simple pole at $t=1$ which has limit $+\infty$ if $Re(t)>1$ and $-\infty$ if $Re(t)<1$, $t \rightarrow 1$.
    \item $h(1/t) = -h(t)$ for all $t \in \mathbb C$ and $\lvert h(t) \rvert > 1$ for all $Re(t)>0, Im(t)<+\infty$
    \item $h(D(0,1)) = \mathbb C_-$
    \item If $b<1$, $h(D(0,b)) = D(c(b),r(b)) \subset \mathbb C_-$
    \item If $b>1$, $\mathbb C \setminus h(D(0,b)) = D(c(b),r(b)) \subset \mathbb C_+$
  \end{enumerate}
  \begin{proof}
    See Appendix~\ref{sec:proofs}.
  \end{proof}
\end{lemma}

We are now in position to define the concept of $\epsilon$-well-behaved gamma convolutions,
  which is a specific regularity assumption on a parametrization of 
  generalized gamma convolutions.
Unfortunately, it is a little cumbersome to describe as Definition~\ref{def:eps_wb} shows.

\begin{definition}[$\epsilon$-well-behaved gamma convolutions]\label{def:eps_wb} 
  A gamma convolution $\mathcal G_{d,n}(\bm \alpha,\bm s)$ is said to be $\epsilon$-well-behaved,
    in short $\epsilon$-w.b.,
    if $\lvert\bm\alpha\rvert>1$ and $\forall I \subseteq \{1,...,n\}$ such that $\sum_{i \in I}\alpha_i > \sum_{i \notin I} \alpha_i$,
    denoting $\bm s_I = \left(\bm s_i\right)_{i \in I}$, 
    for any (complex) solution $\bm t^*$ of the linear system of equations $\bm s_I \bm t = \bm 1$,
    at least for one $j \in 1,...,d$,
    $\lvert h(t_j^*)\rvert > 1+\epsilon$.

  An equivalent statement when $d=1$ is simply that $\lvert\alpha\rvert > 1$ and 
    $\bm s \in ]\frac{\epsilon}{2+\epsilon}, \frac{2+\epsilon}{\epsilon}[^n$.

  If there exists $\epsilon>0$ such that the model is $\epsilon$-w.b.,
    we say that the model is well-behaved, in short w.b.
\end{definition}

Of course, 
  any $\epsilon_1$-w.b. gamma convolution is $\epsilon_2$-w.b. for any $0 < \epsilon_2 < \epsilon_1$.
Property~\ref{prop:epsilon_wb_2} gives equivalent statements for the well-behavior of a gamma convolution,
  hopefully clarifying the requirements from Definition~\ref{def:eps_wb}.

\begin{property}[Well-behaved gamma convolution]\label{prop:epsilon_wb_2}
  For a gamma convolution $\mathcal G_{d,n}(\bm \alpha, \bm s)$ with Thorin measure $\nu$,
    the following statements are equivalent: 
  \begin{enumerate}[label=(\roman*)]
    \item $\mathcal G_{d,n}(\bm \alpha, \bm s)$ is well-behaved.
    \item $\exists\, \epsilon>0$ such that $\mathcal G_{d,n}(\bm \alpha, \bm s)$ is $\epsilon$-well-behaved.
    \item $\lvert\bm\alpha\rvert>1$ and $\forall I \subseteq \{1,\ldots,n\}$ such that 
          $\sum_{i \in I}\alpha_i > \sum_{i \notin I} \alpha_i$, $Ker(\bm s_I) = \{\bm 0\}$.
    \item $\nu(\mathbb R_+^d)>1$ and there exists no $\bm c_1, \ldots,\bm c_p$ linearly independent, $p < d$, such that 
    $$\nu\left(\BigCup\limits_{i=1}^{p} S_{\bm c_i}\right) > \frac{\nu(\mathbb R_+^d)}{2},$$
    where $S_{\bm c}$ is as defined in Property~\ref{prop:special_thorin_measures}.
  \end{enumerate}
  \begin{proof}
    First, 
      $(i) \iff (ii)$ is just the definition of well-behavior.
    Furthermore, 
      $(iv)$ is equivalent to $(iii)$:
      a subset of scales $\bm s_I$ such that $\sum_{i \in I}\alpha_i > \sum_{i \notin I} \alpha_i$ 
      is a subset of scale that has the majority of the weight of the Thorin measure assigned to it,
      and $Ker(\bm s_I) = 0$ if and only if at least $d$ linearly independent directions are inside $\bm s_I$.  
    Therefore, we only show $(ii) \iff (iii)$.

    By Rouché-Capelli's Theorem, see~\cite{suetin1989},
      any complex solution (if there are some) of the linear system 
      $\bm s_I \bm t = \bm 1$ is unique if and only if $Ker(\bm s_I)=\{\bm 0\}$.
    If it exists, call this solution $\bm t^* \in \mathbb C^d$.
    
    Now,
      since for a given $i \in I$, 
      $s_{i,j}$ are all non-negative reals and one of them must not be zero, 
      $\bm t^*$ must have at least one component $t_j^*$ that has a positive real part.
    Furthermore, this component is clearly of bounded modulus since the solution is unique and all $s_{i,j}$ are finite.
    
    Finally,
      the requirement  $\lvert h(t_j^*)\rvert > 1+\epsilon$ is equivalent,
      through Lemma~\ref{lem:mobius}, point $(v)$,
      to $\lvert t_j^* - c(1+\epsilon) \rvert < r(1+\epsilon)$ (where $c$ and $r$ were defined in Lemma~\ref{lem:mobius}).
    
      Since the disk $D(c(1+\epsilon),r(1+\epsilon))$ is in $\mathbb C_+$,
        and tends to the whole $\mathbb C_+$ as $\epsilon$ goes to $0$ by Lemma~\ref{lem:mobius}, point $(iv)$,
        there always exist an $\epsilon >0$ so that a bounded $t_j^*$ with positive real part is inside this disk. 
  \end{proof}
\end{property}

The proof of Property~\ref{prop:epsilon_wb_2} gives a generic way of finding 
  the greatest $\epsilon$ that makes a well-behaved gamma convolution
  $\epsilon$-w.b. for the general multivariate case. 
In a certain sense, the $\epsilon$-w.b. property is a quantification of the statements from Property~\ref{prop:special_thorin_measures}:
  the Thorin measure must not assign half of its mass or more to any subset that would correspond to a non absolutely continuous gamma convolution.
It is therefore stronger than the absolute continuity of the density.
A surprise comes in the next examples, 
  where we show the unexpected fact that the w.b. subclass of gamma convolutions is closed w.r.t finite convolutions.

\begin{example}[Simple w.b. examples]\label{ex:wb} As soon as the total mass of their Thorin measure is greater than one,
  the following gamma convolutions are well-behaved:
  \begin{enumerate}[label=(\roman*)]
    \item Any univariate gamma convolution (the best $\epsilon$ is given in Definition~\ref{def:eps_wb}).
    \item Any independent random vector in $\mathcal G_{d,n}$.
    \item Any invertible linear transformation of a well-behaved (multivariate) gamma convolution
    \item Any finite convolution of well-behaved (multivariate) gamma convolutions.
  \end{enumerate}
  \begin{proof}
    See Appendix~\ref{sec:proofs}.
  \end{proof}
\end{example}

We are now ready to prove the statement of Theorem~\ref{prop:bounds}.

\begin{proof}[Proof of Theorem~\ref{prop:bounds}]
  The theorem stated that, for any $\epsilon, \epsilon'$ s.t. $\epsilon > \epsilon' > 0$ and any dimension $d$,
  there exists a finite positive constant $B(d,\epsilon')$ such that 
  the Laguerre coefficients $\left(a_{\bm k}\right)_{\bm k \in \mathbb N^d}$ of any $d$-variate $\epsilon$-well-behaved gamma convolution verify: 
  $$\lvert a_{\bm k}\rvert  \le B(d,\epsilon')(1+\epsilon')^{-\lvert k \rvert}.$$

  The sketch of the proof is as follows: 
    we construct a multivariate complex function $R$ that has the Laguerre coefficients as its Taylor coefficients,
    and we show by a singularity analysis that this function is analytic around the origin,
    in a polydisc with polyradius $1+\epsilon$.
  We conclude by applying Cauchy's inequality. 
  As the one-dimensional case allows for more detailed computations,
    we start by working for all $d$,
    and split later into the two cases $d=1$ and $d>1$. 
  
  We follow the path of~\cite{kabardov2009} to express the Laguerre coefficients 
    $\left(a_{\bm k}\right)_{\bm k \in \mathbb{N}^d}$ 
    as the Taylor coefficients of a certain ($d$-variate, complex) function $R$.
  From the computation of the moment generating function $M(\bm t)$ of the Laguerre expanded density
    $f(\bm x) = \sum\limits_{\bm k \in \mathbb{N}^d} a_{\bm k} \phi_{\bm k}(\bm x)$,
    we have:
  \begin{align*}
    M(\bm t) = \int\limits_{\mathbb{R}_+^d} e^{\langle \bm t,\bm x \rangle} f(\bm x) \,d \bm x 
    &= \sqrt{2}^{d} \sum\limits_{\bm k \in \mathbb{N}^d} a_{\bm k} \prod\limits_{i=1}^d 
      \sum\limits_{j=0}^{k_i} \binom{k_i}{j} \frac{(-2)^j}{j!} \int\limits_{\mathbb{R}_+} e^{(t_i-1)x_i} x_i^{j}\,d x_i\\
    &= \sqrt{2}^{d} \sum\limits_{\bm k \in \mathbb{N}^d} a_{\bm k} \prod\limits_{i=1}^d 
      \sum\limits_{j=0}^{k_i} \binom{k_i}{j} \frac{(-2)^j}{(1-t_i)^{j+1}}\\
    &= \sqrt{2}^{d} (1 - \bm t)^{-\bm 1}\sum\limits_{\bm k \in \mathbb{N}^d} a_{\bm k} \left(1 - \frac{2}{1 - \bm t}\right)^{\bm k}\\
    &= \sqrt{2}^{d} (1 - \bm t)^{-\bm 1}\sum\limits_{\bm k \in \mathbb{N}^d} a_{\bm k} \left(\frac{\bm t +1}{\bm t -1}\right)^{\bm k},
  \end{align*} 
    which implies that: $$\sum\limits_{\bm k \in \mathbb{N}^d} a_{\bm k} \left(\frac{\bm t +1}{\bm t -1}\right)^{\bm k} 
                           = \sqrt{2}^{-d} (1 - \bm t)^{\bm 1} M(\bm t).$$

  Now denote $\bm h(\bm t) = \left(h(t)\right)_{t \in \bm t}$, 
    so that $\bm h$ applies the Möbius transform $h(t) = \frac{t+1}{t-1}$ from Lemma~\ref{lem:mobius} componentwise. 
  Using the substitutions $\bm y = \bm h(\bm t) \iff \bm t = \bm h(\bm y)$ (since $h$ is its own inverse), 
    which implies that $\left(1-\bm t\right)^{\bm 1} = 2^{d} \left(1-\bm y\right)^{-\bm 1}$,
    we define the function $R$ as:
    \begin{align*}
      R(\bm y) :=\sum\limits_{\bm k \in \mathbb N^d} a_{\bm k} \bm y^{\bm k} 
                &=\sqrt{2}^{-d} M\left(\bm h(\bm y)\right)\left(1 - \bm h(y)\right)^{\bm 1}\\
      &= \sqrt{2}^{d} M\left(\bm h(\bm y)\right)\left(1 - \bm y\right)^{-\bm 1}.
    \end{align*}

  We now study the regularity of the function $R$,
    which is tightly related to the regularity of $M$.
  Recall first the expression of the moment generating function for our random vector:
  $$M(\bm t) = \prod_{i=1}^{n} \left(1 - \langle \bm s_i, \bm t \rangle\right)^{-\alpha_i}$$
  Denote by $\mathcal{V}_{R}$ the singular variety of $R$,
    i.e., the set of points where $R$ is not analytic.
  From the singularities of $\bm y \mapsto (1 - \bm y)^{-\bm 1}$ and from those of  $M$,
    we have:
    \begin{align*}
      \mathcal V_{R} &\subseteq \left\{\bm y :\; \exists j,\; y_j = 1\right\}
      \BigCup \BigCup\limits_{i=1}^n \left\{\bm y:\; \langle \bm s_i,\bm h(\bm y) \rangle = 1\right\}.
    \end{align*}

  Consider first the fact that,
    since the model is $\epsilon$-w.b., $\lvert \bm\alpha \rvert > 1$.
  Hence, 
    $M(\bm h(\bm y))$ dominates $(1 - \bm y)^{-\bm 1}$ when $\bm y$ goes to  $\bm 1$ 
    (as $\bm h(\bm y)$ goes to infinity)
    and therefore $R$ is analytic at $\bm 1$.
  Thus, we only need to consider the singularities of $M(\bm h(\bm y))$,
    and we can further restrict $\mathcal V_R$:
  \begin{equation*}
    \mathcal V_{R} \subseteq \BigCup\limits_{i=1}^n \left\{\bm y:\; \langle \bm s_i,\bm h(\bm y) \rangle = 1\right\}.
  \end{equation*}

  We now split the argument into the cases $d=1$ and $d>1$,
    to show that the $\epsilon$-w.b. condition is equivalent to the analyticity of $R$ on the polydisc with polyradius $1+\epsilon$.

  When $d=1$, 
    all singularities of $R$ are such that $s h(y) = 1$ for a scale $s \in \bm s$.
  In other words, the singularities are at the points $y = h(\frac{1}{s})$ for each scale $s \in \bm s$.
  Recall that $-h(x) = h(x^{-1})$, 
    and that $h$ is (strictly) decreasing before and after its vertical asymptote at $1$.
  Three cases arise:
  \begin{itemize}
    \item If $s \in ]0,1[$,
    then $\lvert y \rvert  = \lvert h\left(\frac{1}{s}\right) \rvert = h\left(\frac{1}{s}\right) > 1+\epsilon 
          \iff \frac{1}{s} < h(1+\epsilon) 
          \iff s > h(1+\epsilon)^{-1} = \frac{\epsilon}{2+\epsilon}$.
    \item If $s \in ]1, +\infty[$,
    then $\lvert y \rvert  = \lvert h\left(\frac{1}{s}\right)\rvert = h(s) > 1+\epsilon 
          \iff  s < h(1+\epsilon) = \frac{2+\epsilon}{\epsilon}$.
    \item If $s = 1$, $h\left(\frac{1}{s}\right) = \pm\infty$ and there is no singularity. 
  \end{itemize}

  Hence, 
    $$\min_{y \in \mathcal V_R}\quad \lvert y \rvert > 1+\epsilon \iff \frac{\epsilon}{2+\epsilon} < \bm s < \frac{2+\epsilon}{\epsilon},$$
  which corresponds (with $\lvert\bm\alpha\rvert>1$) to the one-dimensional $\epsilon$-w.b. condition.

  We now turn ourselves to the multivariate case. When $d > 1$,
    the singularities of $\bm t \mapsto M(\bm t)$ are the zeros of the function:
    $$G\,:\, \bm t \mapsto \prod\limits_{i=1}^n \left(1 - \langle \bm s_i, \bm t\rangle\right)^{\alpha_i}.$$

  Any zero $\bm t^*$ of $G$ must satisfy the system of equations $\bm s_{\bm I} \bm t^* = \bm 1$ 
    for a given non-empty $I \subseteq \{1,...,n\}$.
  However,
    for $\bm t^{*}$ to be unbounded, 
    we must have $\sum_{i =1}^{n} \alpha_i \mathbb 1_{i \in I} \ge \sum_{i =1}^{n} \alpha_i \mathbb 1_{i \notin I}$.

  By the definition of a $\epsilon$-w.b. model,
    there exist no such unbounded zeros as,   
    for any $\bm t^*$ zero of $G$ s.t. $\sum_{i =1}^{n} \alpha_i \mathbb 1_{i \in I} \ge \sum_{i =1}^{n} \alpha_i \mathbb 1_{i \notin I}$,
    we have that: $$\exists i \in 1,...,d, \;\lvert h(t_i^*)\rvert > 1+\epsilon.$$
  
  Thus, 
    the singularities of $R$,
    which are images by $\bm h$ of the zeros of $G$,
    must have at least one dimension that is outside the centered polydisc with radius $1+\epsilon$. 

  Hence,
    whatever $d$,
    $R$ is analytic on the centered polydisc with polyradius $1+\epsilon$.

  Now, 
    for any $d$, 
    by Cauchy's inequality,
    see Theorem 2.2.7 in~\cite{hormander1973},
    by taking a $\epsilon'$ strictly between $0$ and $\epsilon$
    we have the wanted finite upper bound on the coefficients:
    $$B(d,\epsilon') = \sup\limits_{\substack{(\bm \alpha, \bm s): \mathcal G_{d,n}(\bm \alpha, \bm s) \text{ is }\epsilon\text{-w.b.}\\
      \forall i, \lvert y_i \rvert < 1+\epsilon'}} R(y) < \infty.$$
\end{proof}

We can now compute efficiently the Laguerre coefficients of densities in $\mathcal G_{d,n}$,
  with an insurance that they would decrease fast if the model is well-behaved.
We propose in the next subsection to finally discuss an estimation algorithm
  based on an Integrated Square Error loss that leverages these coefficients.

\subsection{Consistency of the induced empirical loss}\label{sec:incomplte_loss}

Expressing the density of $\mathcal{G}_{d,n}$ random vectors into the basis from Definition~\ref{def:laguerre_basis} 
  allows us to truncate the basis and effectively compute an approximated density.
From this,
  by evaluating empirical Laguerre coefficients from data,
  we will fit the parameters to data by minimizing the $L_2$ distance between coefficients (since the basis is orthonormal).

For parameters $\bm \alpha,\bm s$,
  denote by $f_{\bm \alpha,\bm s  }$ the density of the $\mathcal G_{d,n}(\bm \alpha,\bm s  )$ distribution.

For a number of observations $N \in \mathbb N$, suppose $\bm X_1,...\bm X_N$ are i.i.d. random vectors from a true unknown density $f$ with support $\mathbb R_+^d$.

We would like to minimize the integrated square error:
$$\norme{f - f_{\bm \alpha,\bm s  }} =
  \sum\limits_{\bm k \in \mathbb N^d} \left(a_{\bm k}(f) - a_{\bm k}(\bm \alpha,\bm s  )\right)^2.$$

Estimating $a_{\bm k}(f) = \int f(\bm x) \phi_{\bm k}(\bm x) \,d\bm x$ 
by a simple Monte-Carlo plug-in
  \begin{equation}\label{eq:a_k_hat}
    \widehat{a_{\bm k}} = \frac{1}{N}\sum\limits_{i = 1}^{N} \phi_{\bm k}(\bm X_i),
  \end{equation}
  we could compute an approximation of this loss.
However,
  we cannot compute all couples of coefficients $\widehat{a_{\bm k}}$ 
  and $a_{\bm k}(\bm \alpha,\bm s)$ for all $\bm k \in \mathbb N^d$,
  so that we are forced to truncate the basis.

\begin{definition}[$\mathcal G_{d,n}$ parameters estimator] For a set of i.i.d. random vectors $\bm X = (\bm X_1,...,\bm X_N)$,
  we define parameters estimator of a well-behaved $\mathcal G_{d,n}$ distribution that matches the observations as:
  $$\left(\hat{\bm \alpha},\hat{\bm s}\right) =
  \argmin\limits_{\mathcal G_{d,n}(\bm \alpha,\bm s) \,\text{w.b.}} \;L_{\bm m}(\bm X,\bm \alpha,\bm s  ),$$
  where the approximated truncated loss $L_{\bm m}$ is given by:
  $$L_{\bm m}(\bm X,\bm \alpha,\bm s  ) = 
    \sum\limits_{\bm k \le \bm m} \left(\widehat{a_{\bm k}} - a_{\bm k}(\bm \alpha,\bm s  )\right)^2,$$
  and where the empirical Laguerre coefficients $\widehat{a_{\bm k}}$ are given by \eqref{eq:a_k_hat}.
\end{definition}

The loss $L_{\bm m}(\bm X,\bm \alpha,\bm s  )$ can be efficiently computed through Algorithm~\ref{algo:DensityOMGC}.
It will be $0$ if the first coefficients of the Laguerre expansion of $f$ and of our estimator match perfectly,
  i.e.,
  assuming $f \in \mathcal{G}_{d,n}$ and  $\bm m$ big enough,
  if we found the right Thorin measure.
Note that if the goal was given theoretically,
  like in a projection from $\mathcal{G}_d$ to $\mathcal{G}_{d,n}$,
  we could compute $\widehat{a_{\bm k}}$ by formal integration with high precision instead of Monte-Carlo.
We will discuss this option in the next section.
However,
  even in this case,
  the error that comes from the truncation of the basis cannot be computed.

To show the consistency of this loss, we use in the following the results from Theorem~\ref{prop:bounds}.
\begin{property}[Consistency]\label{prop:consistency} Consider that $f \sim \mathcal G_{d,n}(\bm \alpha_0,\bm s_0)$ is well-behaved,
	  and denote $\bm X$ a set of $N$ i.i.d. random vectors with this distribution.
	Fix $\bm m \in \mathbb{N}^d$.
  Suppose that we have a global minimizer
    $$\left(\hat{\bm \alpha},\hat{\bm s}\right) =
      \argmin\limits_{\mathcal G_{d,n}(\bm \alpha,\bm s) \,\text{w.b.}} L_{\bm m}(\bm X,\bm \alpha,\bm s),$$
	  which depends on the threshold $\bm m$ and on the $N$ observations $\bm X$.
	Then:
    $$\lVert f - f_{\hat{\bm \alpha},\hat{\bm s}}\rVert_2^2 \xrightarrow[\substack{N \to \infty\\ \bm m \to \infty}]{a.s} 0.$$

\begin{proof} 
  To show the result, 
    we start by expressing the loss in the Laguerre basis, 
    and we split the basis on indices smaller and greater than $\bm m$.
  For the sake of simplicity, we will omit $f$ in $a_{\bm k}(f)$ and denote $a_{\bm k} = a_{\bm k}(f)$. We have:
	\begin{align*}
		\lVert f - f_{\hat{\bm \alpha},\hat{\bm s}}\rVert_2^2 
    &= \sum\limits_{\bm k \in \mathbb N^d} \left(a_{\bm k} - a_{\bm k}(\hat{\bm \alpha},\hat{\bm s})\right)^2\\
		&= \underbrace{\sum\limits_{\bm k \le \bm m} \left(a_{\bm k} - a_{\bm k}(\hat{\bm \alpha},\hat{\bm s})\right)^2}_{A} + 
		   \underbrace{\sum\limits_{\bm k > \bm m} \left(a_{\bm k} - a_{\bm k}(\hat{\bm \alpha},\hat{\bm s})\right)^2}_ {B},
	\end{align*}
  where $\bm k > \bm m \iff \exists i \in 1,...,d :\; k_i > m_i$.

  We first show that $A \xrightarrow[N \to \infty]{a.s.} 0$ for any fixed $\bm m$. 
  For $A$,
    we can treat each term in the summation independently as the sum is finite.
	$A$ can be then expanded further by plugging-in the Monte-Carlo estimator
  $\widehat{a_{\bm k}} = \frac{1}{N}\sum\limits_{i = 1}^{N} \phi_{\bm k}(\bm X_i)$ and expanding the square:
	\begin{align*}
		A &= \underbrace{\sum\limits_{\bm k \le \bm m} \left(a_{\bm k} 
            - \widehat{a_{\bm k}}\right)^2}_{A_1} +
			   \underbrace{2\sum\limits_{\bm k \le \bm m} \left(a_{\bm k} 
            - \widehat{a_{\bm k}}\right)\left(\widehat{a_{\bm k}} 
            - a_{\bm k}(\hat{\bm \alpha},\hat{\bm s})\right)}_{A_2} +
         \underbrace{\sum\limits_{\bm k \le \bm m} \left(\widehat{a_{\bm k}} 
            - a_{\bm k}(\hat{\bm \alpha},\hat{\bm s})\right)^2}_{L_{\bm m}(\bm X,\hat{\bm \alpha},\hat{\bm s})}.
	\end{align*}

	Now,
	  since $\widehat{a_{\bm k}} \xrightarrow[N \to \infty]{a.s.} a_{\bm k}$ by the strong law of large numbers,
	  $A_1 \xrightarrow[N \to \infty]{a.s.}0$.

	Furthermore,
	  since we restricted the optimization to well-behaved estimators,
    each $a_{\bm k}(\hat{\bm \alpha},\hat{\bm s})$ is bounded by Theorem~\ref{prop:bounds}.
  Now $\lvert \widehat{a_{\bm k}} \rvert \le \sqrt{2}^d$ because $\lvert \phi_{\bm k}(\bm x) \rvert \le \sqrt{2}^d$
    whatever $\bm x \in \mathbb R_+^d$.
	Therefore, 
	  for any $\bm k$, 
	  $\left(\widehat{a_{\bm k}} - a_{\bm k}(\hat{\bm \alpha},\hat{\bm s})\right)$ is bounded
    and $\left(a_{\bm k} - \widehat{a_{\bm k}}\right) \xrightarrow[N \to \infty]{} 0$, 
	  which makes $A_2 \xrightarrow[N \to \infty]{a.s.} 0$.

	Last but not least,
	  by definition of $\left(\hat{\bm \alpha},\hat{\bm s}\right)$ as global minimizers,
	  $$L_{\bm m}(\bm X,\hat{\bm \alpha},\hat{\bm s}) \le\allowbreak L_{\bm m}(\bm X,\bm \alpha_0,\bm s_0) \xrightarrow[N \to \infty]{a.s.} 0,$$ 
    since Monte-Carlo estimators are not biased, and $\left(\bm \alpha_0,\bm s_0\right)$ is inside the range of optimizations.
	Therefore,
    we can conclude that $A \xrightarrow[N \to \infty]{a.s.} 0$ for any $\bm m$ fixed. 

	Now consider the remainder $B$.
  We show that $B \xrightarrow[\bm m \to \infty]{a.s.} 0$ uniformly in $N$. 
  Since both the true model and the approximation are well-behaved, 
  there exists $\epsilon > \epsilon' > 0$ so that they are both $\epsilon$-w.b.
By Theorem~\ref{prop:bounds} we have then,
  \begin{align*}
    \lvert a_{\bm k}\rvert &\le B(d,\epsilon') (1 + \epsilon')^{-\lvert\bm k\rvert}\\
    \text{and }\lvert a_{\bm k}(\hat{\bm \alpha},\hat{\bm s})\rvert &\le B(d,\epsilon') (1 + \epsilon')^{-\lvert\bm k\rvert}.
  \end{align*}
  Therefore:
  \begin{align*}
    B &= \sum\limits_{\bm k > \bm m} \left(a_{\bm k} - a_{\bm k}(\hat{\bm \alpha},\hat{\bm s})\right)^2\\
      &\le 4 B(d,\epsilon')^2 \sum\limits_{\bm k > \bm m} \left((1+\epsilon')^2\right)^{-\lvert \bm k\rvert} \xrightarrow[\bm m \to \infty]{a.s.} 0,
  \end{align*}
    uniformly in $N$,
    which concludes the argument.
\end{proof}
\end{property}

Although the consistency result restricts the parameter space to only well-behaved Thorin measures, 
  in practice we run our optimization procedures without this restriction,
  and simply check at the end that the resulting estimators are well-behaved.
When $d=1$, we can always fix $\epsilon$ to make the estimated gamma convolution $\epsilon$-w.b. if $\lvert\hat{\bm\alpha}\rvert>1$.
In practice, 
  numerical computations have troubles producing collinear scales, 
  so every estimated multivariate gamma convolution is usually well-behaved
  as soon as the total mass of the produced Thorin measure is greater than $1$.
Recall that univariate distributions like log-Normal and Pareto have an infinite total mass for the Thorin measure.

We now have a reliable loss to minimize and an efficient algorithm to compute it,
  allowing us to estimate multivariate gamma convolutions from data.
To test the approach,
  we propose in the next section to discuss some potential applications.

\section{Investigation of performance}\label{sec:investigation}

In this section,
  we show the results of our algorithm on several examples,
  both in univariate and multivariate settings.

Note that,
  through the recursive Faà di Bruno's formula in Algorithm \ref{algo:DensityOMGC},
  we produce Laguerre coefficients as high degree multinomials into
  $\bm \alpha,\mu_{\bm 0}$ and $\bm \xi$ (the simplex projection of the scales).
The loss will therefore have a myriad of local minima,
  making it non-convex and forcing us to use global minimization routines.
We found the Particle Swarm~\cite{koyuncu2019,zong-junli2016} global optimization routine
  to work particularly well.

All the implementation was gathered in the provided Julia package,
  along with the code corresponding to this investigation.
Indeed, 
  Julia allowed the use of high-precision arithmetic together 
  with compiled code and global optimizations routines without the (potentially error-prone)
  reimplementation of at least one of these three features\footnote{For more details about Julia, see julialang.org.},
  which was not possible in Python, R, Matlab or C++ for example.

We start by the projection of known densities onto the $\mathcal G_{1,n}$ class,
  differentiating the cases when the known density is inside $\mathcal G_{1}$ or not.
Then we discuss the estimation of $\mathcal G_{1,n}$ models from empirical data,
  with heavy tailed simulated data,
  and we finally look at the estimation of $\mathcal{G}_{d,n}$ models from multivariate empirical data,
  with a particular emphasis on the properties of the dependence structure underlying the data.
We conclude by a direct application on a real dataset, using the \emph{Loss-Alae} dataset from Klugman \& Parsa~\cite{klugman1999}.

\subsection{Projection from a known density}\label{sec:projection_experiemnt}

The two first examples are a log-Normal and a Weibull, 
  that have the particularity of being respectively inside and outside $\mathcal G_{1}$.
We will also take a look at the Pareto case, 
  which is inside $\mathcal{G}_1$,
  but might be outside $L_2(\mathbb R_+)$.

We want to compare our algorithm to the MFK procedures, discussed in Section~\ref{sec:projection_furman} on fair grounds.
Therefore,
  we slightly modify our algorithm to use coefficients based on the \emph{theoretical} integrals $\mu_{k,-1}$,
  as MFK.
The choice of the input distributions is also guided by~\cite{miles2019,furman2017}.
The experiment is as follows:

\begin{itemize}
	\item We compute shifted moments $\left(\mu_{k,-1}\right)_{k \le 2n}$
    of the distribution through direct integration of the density,
	  at 1024 bits of precision (to match the 300 digits that MFK requires).
	\item We use MFK to fit a $\mathcal{G}_{1,n}$ based on these moments.
	\item Through Eq \ref{eq:laguerre_coef_expression},
	  we compute Laguerre coefficients based on the same shifted moments $\mu_{k,-1}$,
	  and we minimize the loss given through Algorithm \ref{algo:DensityOMGC} to obtain our approximation.
	\item We compute Kolmogorov-Smirnov (one-sample, exact) distances between $N=10\,000$ simulations
    from the estimated $\mathcal G_{1,n}$ model to the theoretical distribution,
	  on $B=100$ different simulations.
\end{itemize}

\subsubsection{Projection from \texorpdfstring{$\mathcal{G}_1$}{G1} densities to \texorpdfstring{$\mathcal{G}_{1,n}$}{G1n}}\label{sec:log-Normal-projection}
The log-Normal distribution is a fundamental distribution in the $\mathcal{G}_{1}$ class.
The proof that it belongs to $\mathcal{G}_1$ is actually what historically initiated the study of the class by Thorin.
To match MFK's experiment,
  we conduct our first comparison with a log-Normal $LN(\mu=0,\sigma=0.83)$ density.
The resampled Kolmogorov-Smirnov distances are summarized,
  for several $n$ and for both algorithms,
  in Figure~\ref{fig:furman_ln_violin}.
\begin{figure}[H]
	\centering
	\includegraphics[width=\textwidth]{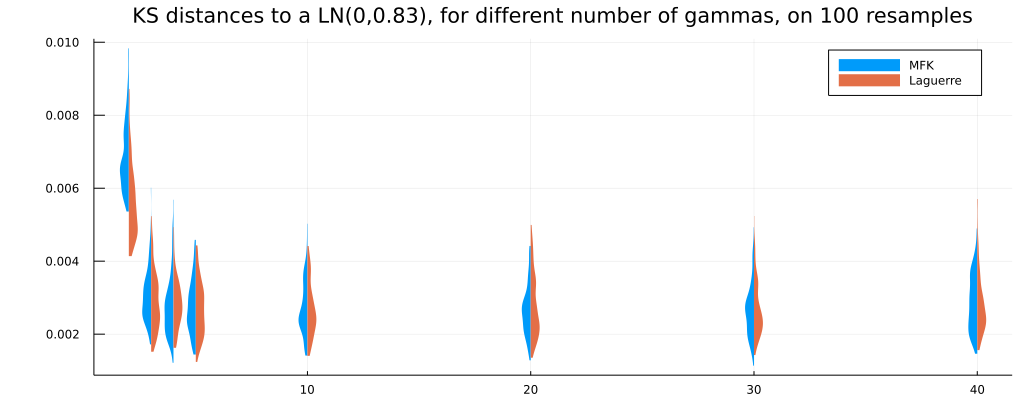}
	\caption{Violin densities of resampled KS distances (smaller is better).
			     In abscissa is $n \in \{2,3,4,5,10,20,30,40\}$, the number of gammas used in the approximation.}%
	\label{fig:furman_ln_violin}%
\end{figure}
Although MFK has a very strong convergence result,
  for $n=2$ gammas our estimator performs better with the same information 
  (theoretical shifted moments given with $1024$ bits precision),
  and that for $n > 2$,
  the performance is overall the same.
Indeed,
  we see that in most cases the two algorithms found parameters that are close to each other.
We reported the shapes and scales that both algorithms produced for $n=2,3,4$ and $5$ in Table \ref{tab:furman_compare}.
\begin{table}[H]
  \scriptsize
	\caption{\label{tab:furman_compare}Estimated parameters from both algorithms on the  log-Normal example}
	\centering
	\begin{tabu} to \linewidth {>{\raggedright}X>{\centering}X>{\centering}X>{\centering}X>{\centering}X}
	\toprule
	\multicolumn{1}{c}{ } & \multicolumn{2}{c}{MFK} & \multicolumn{2}{c}{Laguerre} \\
	\cmidrule(l{3pt}r{3pt}){2-3} \cmidrule(l{3pt}r{3pt}){4-5}
	 & $\hat{\bm \alpha}$ & $\hat{\bm s}$ & $\hat{\bm \alpha}$ & $\hat{\bm s}$\\
	\midrule
	\addlinespace[0.3em]
	\hspace{1em}\textbf{n=2} & 0.5688 & 1.5622 & 0.5458 & 1.6283\\
	\hspace{1em} & 2.4596 & 0.1941 & 2.4539 & 0.1999\\
	\addlinespace[0.3em]
	\hspace{1em}\textbf{n=3} & 0.2195 & 2.4902 & 0.2070 & 2.5781\\
	\hspace{1em} & 0.8934 & 0.6718 & 0.8919 & 0.6875\\
	\hspace{1em} & 2.6081 & 0.0972 & 2.6071 & 0.0987\\
	\addlinespace[0.3em]
	\hspace{1em}\textbf{n=4} & 0.0919 & 3.4076 & 0.0844 & 3.5307\\
	\hspace{1em} & 0.4596 & 1.2208 & 0.4555 & 1.2513\\
	\hspace{1em} & 1.0042 & 0.3737 & 1.0063 & 0.3792\\
	\hspace{1em} & 2.6956 & 0.0588 & 2.6957 & 0.0594\\
	\addlinespace[0.3em]
	\hspace{1em}\textbf{n=5} & 0.0398 & 4.3466 & 0.0346 & 4.5447\\
	\hspace{1em} & 0.2569 & 1.7963 & 0.2492 & 1.8576\\
	\hspace{1em} & 0.5708 & 0.7234 & 0.5721 & 0.7394\\
	\hspace{1em} & 1.0576 & 0.2399 & 1.0609 & 0.2428\\
	\hspace{1em} & 2.7574 & 0.0396 & 2.7582 & 0.0399\\
	\bottomrule
	\end{tabu}
\end{table}

Although the produced estimations are close to each other on this example,
  our method has several clear benefits:
  the full density does not need to belong to $\mathcal{G}_1$,
  and we do not require shifted moments up to $300$ digits precision.
In the next subsection,
  we describe the same experiment on a Weibull density that is known to make MFK simply fail,
  even with accurate enough computations.

\subsubsection{Projection from densities outside \texorpdfstring{$\mathcal{G}_{1}$}{G1} to \texorpdfstring{$\mathcal{G}_{1,n}$}{G1n}}

Let $X$ be a Weibull r.v. with shape $k>0$.
Its pdf can then be written as:
  $$f(x) = k x^{k-1} e^{-x^{k}} \bm 1_{x \ge 0}.$$
	
When $k > 1$,
  this distribution does not belong to $\mathcal{G}_1$.
In~\cite{miles2019,furman2017},
  authors report negative shapes for the MFK approximation in $\mathcal{G}_{1,2}$ of the Weibull with shape $k = \frac{3}{2}$,
  which is clearly a failure.
This failure is occurring because the truncated moment problem that MFK is solving has no solution,
  as we detailed in Section \ref{sec:projection_furman},
  since the (normalized) cumulants  $\bm \xi$ are not inside the moment cone.

However,
  the Laguerre basis is an orthonormal basis of $L_2(\mathbb R_{+})$,
  which contains this density.
The projection procedure detailed at the top of Subsection~\ref{sec:projection_experiemnt} still works correctly:
  it produces meaningful shapes and scales that are positive,
  whatever the number of gammas we ask for.
Figure~\ref{fig:weibull_projection_boxplot} displays a Violin plot of KS distances for several numbers of gammas:

\begin{figure}[H]
	\centering
	\includegraphics[width=\textwidth]{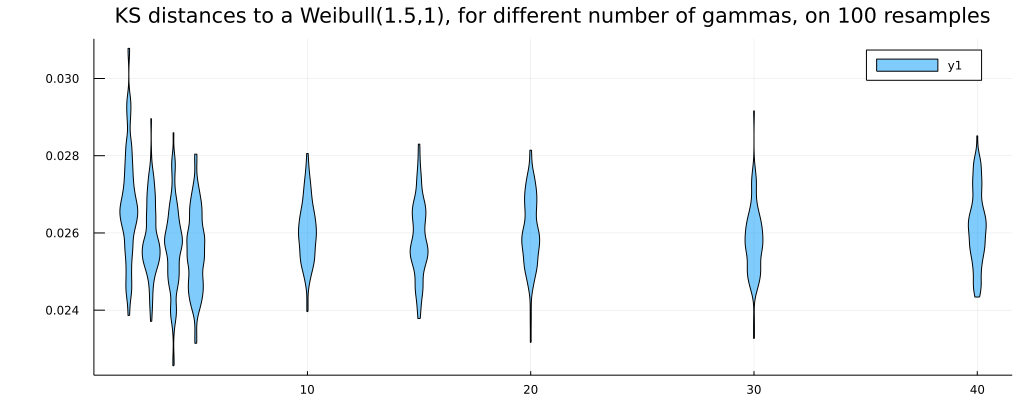}
	\caption{Violin densities of resampled KS distances for approximation of a Weibull($k=1.5$).
	In abscissa is $n \in \{2,3,4,5,10,20,30,40\}$, the number of gammas used in the approximation.
  Only the Laguerre approximation results are presented.}%
	\label{fig:weibull_projection_boxplot}%
\end{figure}

On Figure~\ref{fig:weibull_projection_boxplot},
  we see that the KS distance is decreasing for an increasing number of gammas $n$,
  when $n$ is small,
  but reaches a minimum quite quickly.
Indeed,
  the distance is lower-bounded since the target is not in the class.
Although none of the proposed approximations passes the KS test,
  we will revisit this example later from an estimation point of view to show that
  the projection is still an acceptable approximation of the input Weibull distribution.

Note that the KS distance is not stable for increasing $n$:
  the algorithm has troubles to fix parameters to $0$ if needed.
An additional penalty term in this direction might be a good solution to overcome this behavior.

Through these two first examples,
  we showed the projection possibilities of our algorithm,
  and compared them with MFK,
  the only existing method in the literature.
We saw that for input densities inside the $\mathcal G_1$ class our method performs as good as MFK.
Moreover,
  our method still works properly and produces meaningful results for examples outside the class,
  contrary to MFK which does not give any result.
We will now talk about estimation from data,
  a problem that has never been solved for gamma convolutions,
  and showcase how our algorithm can find a good representation for a given dataset,
  whatever the dimension of the dataset.

In the following,
  we will treat empirical datasets given in standard \texttt{IEE 745 float64} precision.

\subsection{Estimation from univariate empirical data}

Before coming back in more details to the Weibull case,
  we wanted to perform some tests with empirical data coming from simulations of two heavy-tailed distributions: 
  log-Normal and Pareto.

\subsubsection{Heavy-tailed examples}

We consider first a dataset of $100'000$ samples from a log-Normal distribution
  with the same parameter values as in Section~\ref{sec:projection_experiemnt},
  $LN(\mu = 0,\sigma = 0.83)$,
  in $64$ bits precision.

From this sample,
  we compute empirical Laguerre coefficients,
  and then we optimize the set of shapes and scales of the estimated density to minimize
  the truncated $L_2$ loss between them and the ones produced through Algorithm \ref{algo:DensityOMGC}.
Following~\cite{furman2017,miles2019},
  we choose the size of the basis as being $m=2n+1$,
  such that if we remove the first moment that is irrelevant,
  we have a number of moments equal to the number of parameters\footnote{Since Julia starts counting at $1$ and not $0$,
    we actually have $m$ moments and not $m+1$ as we did in the previous theoretical analysis.}.
We ran the experiment for different number of gammas $n = 10,20$ and $40$.
The corresponding results are respectively available in Figure~\ref{fig:log-Normal10} for $n=10$,
  and in Figures~\ref{fig:log-Normal20} and \ref{fig:log-Normal40} for $n=20,40$, see Appendix~\ref{sec:figs_and_tables}.

\begin{figure}[H]
  \centering
  \includegraphics[width=\textwidth]{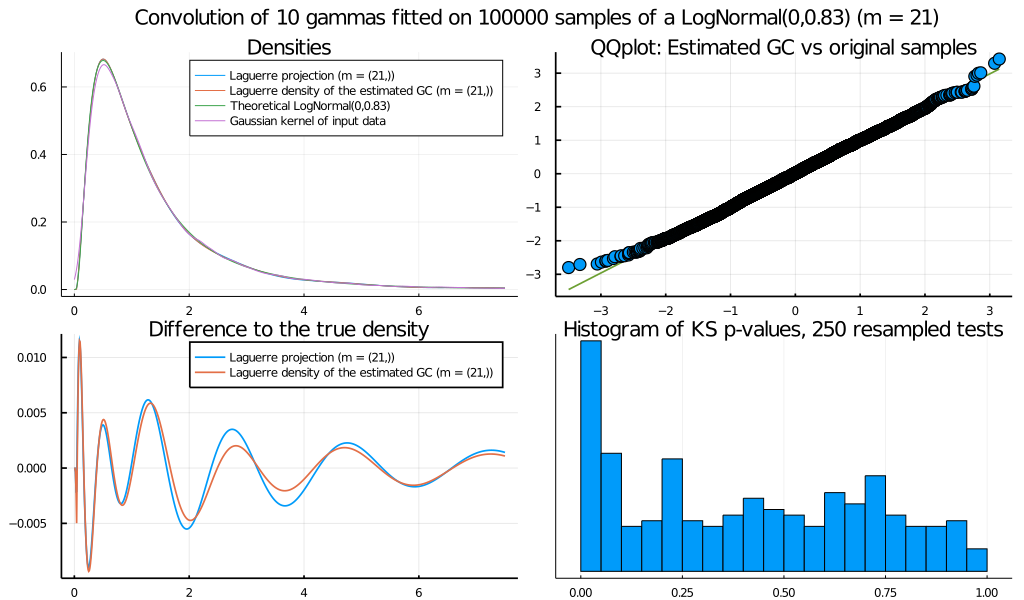}
  \caption{Log-Normal results with $10$ gammas:
        Upper left:
          the comparison of the density approximations.
        Lower left:
          the difference of to the true density.
        Upper right:
          a quantile-quantile plot of the approximation against the true distribution.
        Lower right:
          p-values of one-sample KS tests of simulations from the estimator against the true distribution.}%
  \label{fig:log-Normal10}%
\end{figure}

We see on these graphs that the approximation is good enough to fool the Kolmogorov-Smirnov test.
The absolute pointwise error between the Laguerre approximation of the density of the gamma convolution
  and the true log-Normal density is lower than $0.005$ for $n=10,20$,
  and twice smaller for $n=40$.
A positive but surprising result is that,
  although the distance between the Laguerre approximation and the true density does not reduce with increasing $m$,
  the distance between the gamma approximation and the true density does.

The KS test setup is as follows:
we simulate $250$ resamples of $100'000$-sized dataset from the estimated density,
  and compare it with a one-sample exact KS test to the true density (and not to the samples that the estimator was computed from).
For all values of $n$,
  we observe close-to-uniform distribution of the p-values,
  which is a good result.

We also treated several datasets simulated from Pareto's distribution,
  with exactly the same experiment.
Pareto's distribution has a shape parameter $k > 0$ that influences the thickness of the tail:
  the distribution has a variance if and only if $k > 2$,
  an expectation if and only if $k > 1$ and the density belongs to $L_2$ if and only if $k > 1/2$.
Note that these inequalities are strict. 

The fact that the distribution has no variance or no expectation is not really a problem for our procedure,
  since we use only shifted moments that all Pareto distributions have,
  but our convergence result requires a square-integrable density.
Therefore,
  we propose to check several values of the shape parameter $k$,
  covering all cases,
  while varying the number $n$ of gammas.
Figure~\ref{fig:Pareto_summary} summarizes the results by only showing 
  quantile-quantile plots of the estimated distributions against the starting Pareto distribution.
All the models follow the same procedure as for the previous log-Normal estimators.

\begin{figure}[H]
	\centering
	\includegraphics[width=\textwidth]{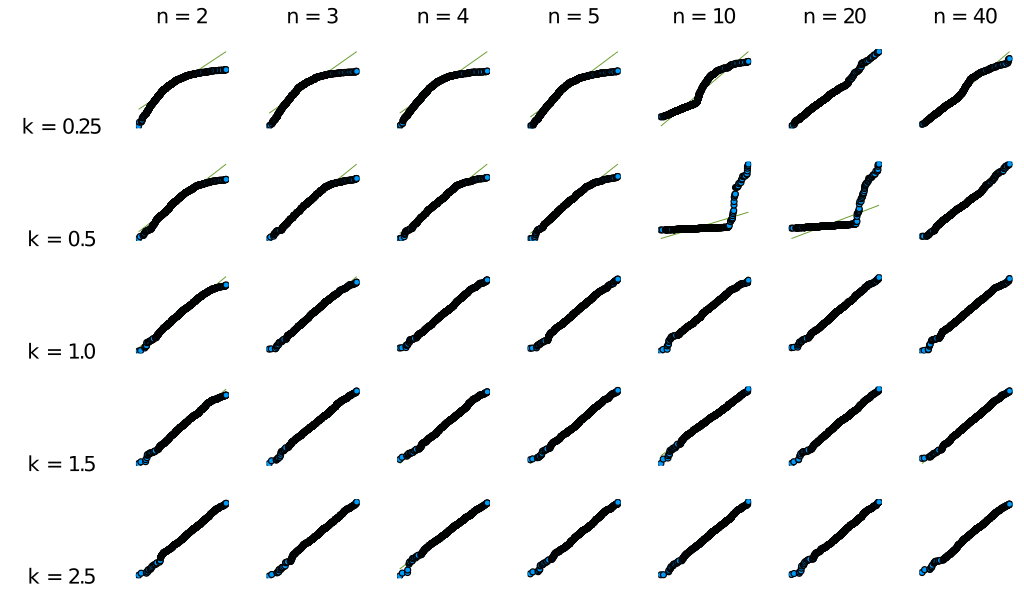}
	\caption{Quantile-quantile plots for Pareto experiments,
             with $N=1000$ samples (log-scaled).
           We only show the $995$ first points:
             $5$ points are excluded in the tail for clarity.
           Each row corresponds to a shape parameter for the Pareto,
             and each column to a number of gammas.}%
	\label{fig:Pareto_summary}%
\end{figure}

Note that even quantile-quantile plots between two simulations from the same Pareto distribution
  have a tendency to deviate in the extremes (due to the fat tail),
  which is why we excluded the last points.
In each case,
  we also ran the same Kolmogorov-Smirnov testing procedure as for the  log-Normal,
  and obtained very good uniformity of the p-values,
  except for k = 0.25 when $n$ is too small.
We see on the graph that increasing the number of gammas usually produces better results,
  but also that the heavier the tail (smaller shapes) the harder it is to get a good fit.

We see on these graphs that the method is accurate enough for reasonably high quantiles,
  and very good in the body of the distribution.
The $k = 0.25$ case is not in the theoretical boundaries for the convergence result to take place,
  we have to wait until at least $20$ gammas to have an acceptable fit.

\subsubsection{An example outside the class}

We revisit the Weibull($k=1.5$) case of Section~\ref{sec:projection_experiemnt},
  a distribution that is not inside the class.
We ran our algorithm on $100'000$ samples from a Weibull with shape $k = \frac{3}{2}$,
  with successive numbers of gammas of $2,10,20$ and $40$.
The corresponding graphs are shown in Figure~\ref{fig:weibull1.5_rez}.
Here,
  we show only the difference on the density estimation and the quantile-quantile plot 
  as one notable thing is that the Laguerre expansion of the density gets better and better 
  as we increase the size of the basis (jointly with the number of gammas),
  but the estimator does not:
the difference to the density is always the same,
  and always makes Kolmogorov-Smirnov reject the proposed estimator,
  although quantile-quantile plots indicate correct fits.

\begin{figure}[H]
	\centering
	\includegraphics[width=\textwidth]{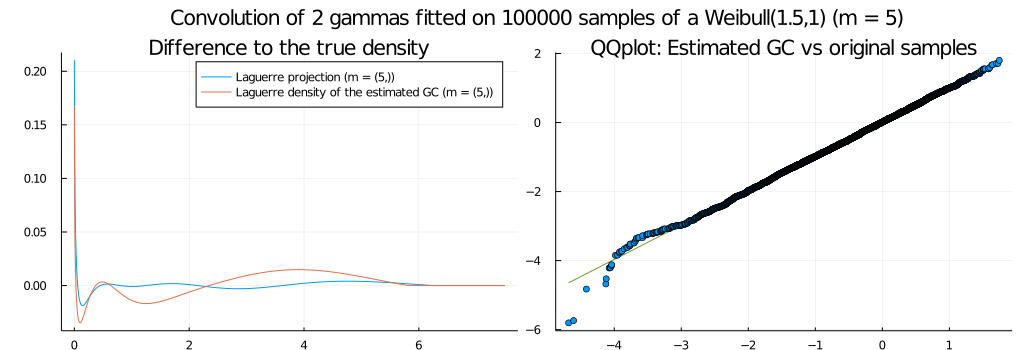}\\
	\includegraphics[width=\textwidth]{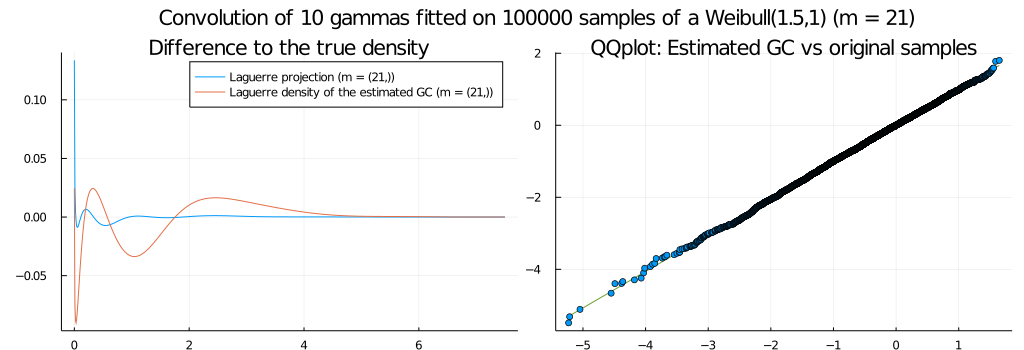}\\
	\includegraphics[width=\textwidth]{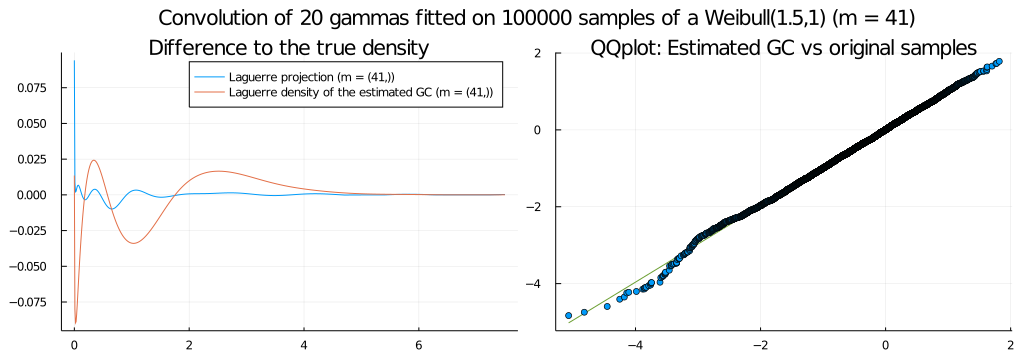}\\
	\includegraphics[width=\textwidth]{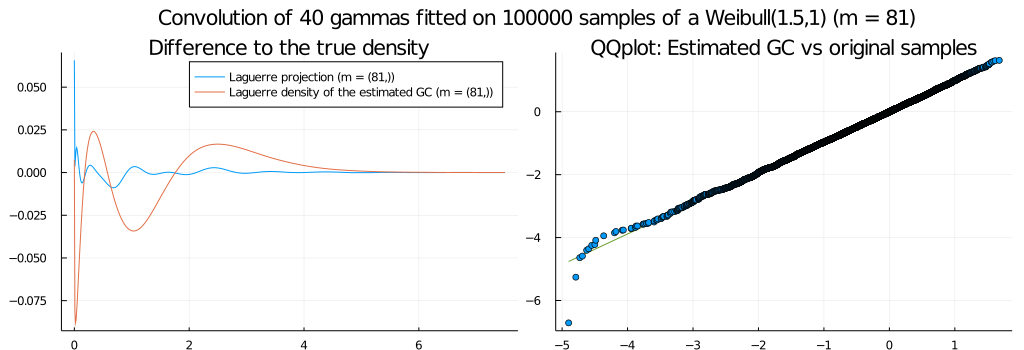}
	\caption{Weibull results with $2,10,20$ and $40$ gammas.
Same legend as Figure~\ref{fig:log-Normal10}.}%
	\label{fig:weibull1.5_rez}%
\end{figure}

On these estimations,
  several things should be noted:

\begin{itemize}
	\item The KS tests reject the hypothesis that the distributions are the same whatever the number of gammas,
	  which was expected.
	\item The pointwise difference on the density of the estimator does not match the error of the Laguerre expansion
    (which is getting smaller as the basis grows) but stays the same whatever the size of the basis and the number of gammas.
	This error represents some kind of distance between the Weibull distribution and the class.
	\item The quantile-quantile plot however indicates a correct fit on all the distribution except maybe the left tail.
\end{itemize}

These three points show that our procedure works even for data that are far from the $\mathcal{G}_1$ class,
  finding some orthogonal projection,
  relative to the basis,
  onto the class.
This is a very good point for the procedure.

In the next examples,
  we will showcase some multivariate uses of the procedure,
  with a particular emphasis on the dependence structure of the simulated data and of the produced estimators.

\subsection{Estimation from multivariate empirical data}

We now discuss multivariate capabilities of our algorithm.
As no known algorithm estimates multivariate gamma convolutions,
  we are not able to compare our results to other procedures.
However,
  we can still check quantile-quantile plots for marginals and 
  verify that the shape of the dependence structure matches the inputted one.

As use-cases might vary,
  we propose two different examples:
  a multivariate log-Normal,
  with a Gaussian copula that induces no tail dependency,
  and a (survival) Clayton copula,
  including a strong tail dependency.
Note that we reduced the exposition to estimation from empirical data,
  but formal integration to obtain $\left(\mu_{\bm k,
  -\bm 1}\right)_{\bm k \le \bm m}$
  from a formal density at any desired computational precision is also a perfectly fine use case.
However, we do not have a clever way of choosing the parameter $\bm m$ yet.
We found that setting all $m_i$'s equal to $n$ provide enough evaluation points to obtain a good fit on our simple examples.

\subsubsection{A Gaussian example}

We do not know if the $d$-variate log-Normal,
  defined as the componentwise exponentiation of a $d$-variate Gaussian random vector,
  is inside $\mathcal{G}_d$.
However,
  we can still run our procedure and observe the resulting approximations.

A first result is given in Figure~\ref{fig:MLN1},
  that corresponds to a $\mathcal{G}_{2,10}$ estimation taken from $100'000$ data points simulated from a
  $MLN(\bm \mu = \bm 0,\bm \sigma = 1,\rho = 0.5)$.
Figure~\ref{fig:MLN1} proposes quantile-quantile plots for the marginals,
  and a simple Gaussian kernel density estimator for the pseudo-observations on the original and estimated models.

\begin{figure}[H]
	\centering
	\includegraphics[width=0.95\textwidth]{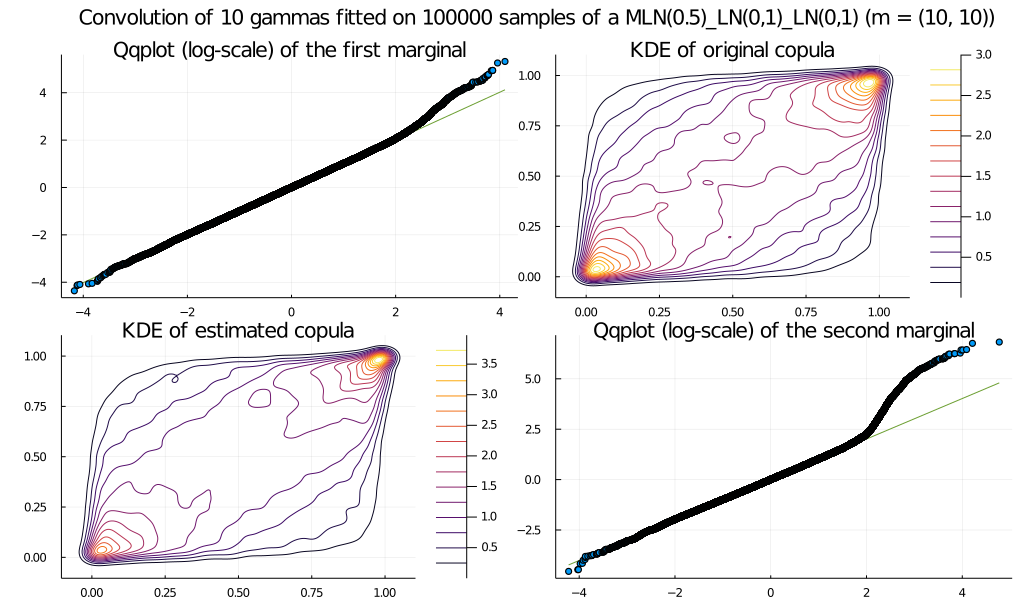}
	\caption{Multivariate log-Normal results with $10$ gammas. 
  Top left and bottom right: 
  quantile-quantile plots for
  the two marginals. 
Bottom left and top right: 
  levels plots of pseudo-observations from simulations from,
  respectively,
  the estimated model and the dataset.}%
	\label{fig:MLN1}%
\end{figure}
We see on Figure~\ref{fig:MLN1} that the dependence structure was correctly reproduced by
  the approximation as a convolution of $10$ (comonotonic) gamma random vectors,
  but marginal tails are not estimated correctly.
Thankfully,
  pushing the number of multivariate gammas to $n=20$ solves this issue,
  as shown in Figure~\ref{fig:MLN2}:
\begin{figure}[H]
	\centering
	\includegraphics[width=0.95\textwidth]{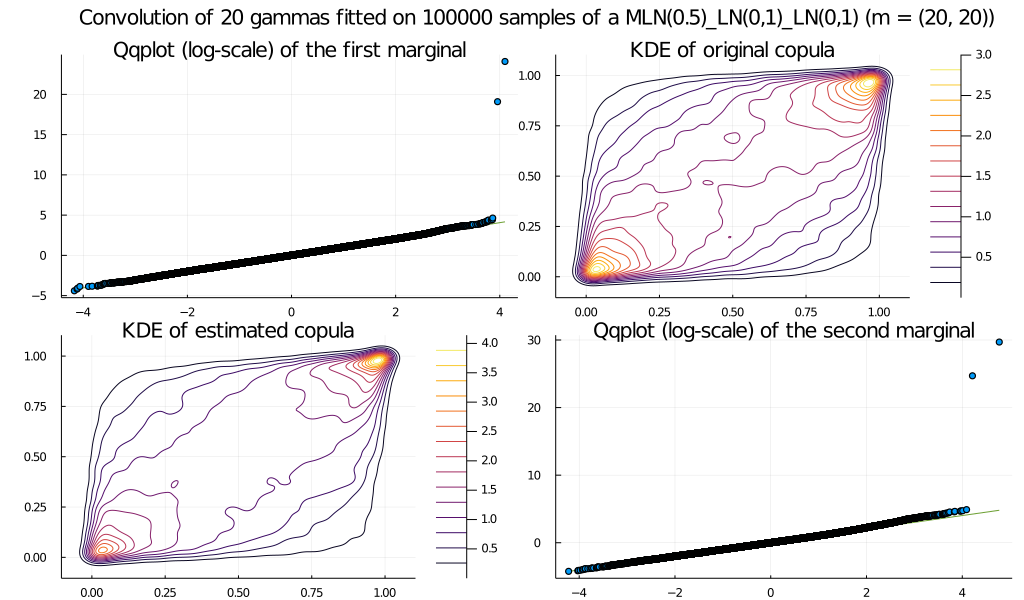}
	\caption{Multivariate log-Normal results with $20$ gammas. Same legend as Figure~\ref{fig:MLN1}.}%
	\label{fig:MLN2}%
\end{figure}

On Figure~\ref{fig:MLN2},
  we see that the dependence structure is still very good,
  and that the problem of the marginal tails has been solved by adding a few more gammas.
Note that since we have $100'000$ points in the quantile-quantile plots,
  one point out of the line corresponds to a deviation at the $0.00001$-quantiles,
  which are of little importance in many use cases.

Overall,
  marginal quantile-quantile plots are correct when the number of gammas is big enough,
  and the Gaussian dependence structure is well-fitted.
The Gaussian copula is a copula that exhibits no tail dependency.
We already know that our model can exhibit asymmetric behaviors,
  but the possibility of tail dependency is also an important feature to have.
Therefore,
  in the next example we run the model on datasets that have such a feature.

\subsubsection{A copula with tail dependency}

We fit data whose dependence structure is given through a (survival) Clayton copula.
This copula is a symmetric copula that exhibits upper tail dependency.
We took here a parameter $\theta = 7$,
  yielding a strong positive dependence structure.

In Figure~\ref{fig:clayton20},
  you can see the results corresponding to data sampled from a Clayton copula,
  a Pareto(k = 2.5) marginal and a log-Normal $LN(0,0.83)$ marginal.
The conditions of the experiments are identical to the previous subsection.

\begin{figure}[H]
	\centering
	\includegraphics[width=0.95\textwidth]{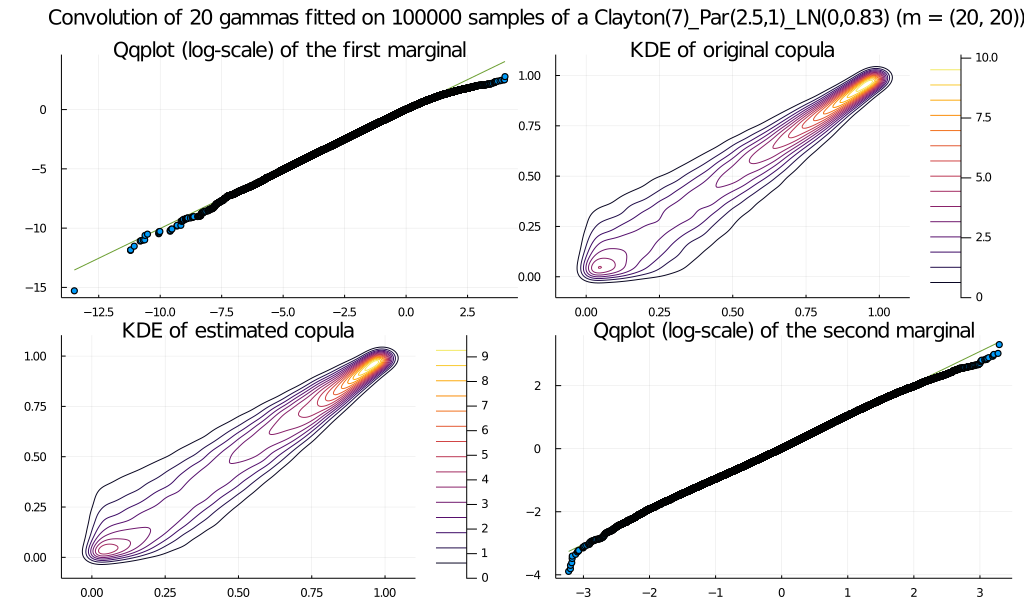}
	\caption{Results with $20$ gammas.
			 Same legend as Figure~\ref{fig:MLN1}.}%
	\label{fig:clayton20}%
\end{figure}

We note that marginal upper tails are not perfectly fitted:
  this could probably be overcome by taking a greater number of gammas.
However,
  we see that the shape of the estimated copula is not perfectly symmetric:
  there is no constraint on the model to produce a symmetric dependence structure,
  and therefore producing one is hard.
The small misfit around the $(0.25,0.02)$ quantile is probably due to this:
  marginals are not on the same scale,
  and therefore the Laguerre coefficients with $k_1 \gg k_2$ are probably on a different scale than those with $k_2 \gg k_1$.
The goodness of fit is therefore not symmetric.

\subsubsection{A real dataset}
To test the algorithm on more realistic data,
  we chose the Loss-Alae dataset from Klugman \& Parsa~\cite{klugman1999}.
This dataset features $1500$ observations of two variables: “Loss” and “Alae”.
We refer to ~\cite{klugman1999} for description of this (quite standard) dataset
  in insurance valuations and copula estimations fields.

\begin{figure}[H]
	\centering
	\includegraphics[width=\textwidth]{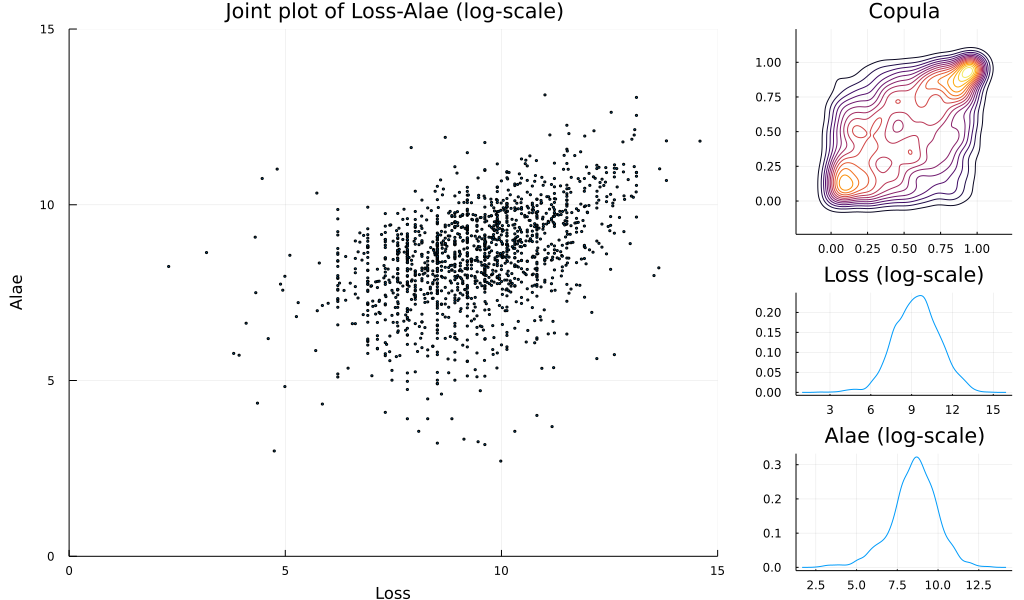}
	\caption{The original Loss-Alae dataset: On the left, a log-scale joint plot of the dataset. 
  On the right, from top to bottom, Gaussian kernel density estimates of the copula and of the two marginals (both log-scale).}%
	\label{fig:lossalae_data}%
\end{figure}

On Figure~\ref{fig:lossalae_data},
  we see that the dependence is somewhat strong between the two marginals.
There is left and right tail dependency, 
  and the marginals are clearly heavy-tailed.

We ran the algorithm with numbers of gammas $n \in (2,3,4,5,10,20)$. 
We also fixed $\bm m = (n,n)$, arbitrarily.
The obtained shapes and rates are given in Table~\ref{tab:loss_alae_comparaison}, in Appendix~\ref{sec:figs_and_tables}. 
We note that all estimators are $\epsilon$-well-behaved, 
  and that, 
  sometimes, 
  the algorithm uses zero scales.
Quantiles-quantiles plots of the marginals and kernel density estimation 
  of the dependence structures for $n=2,5,10,20$ are reported in Figure~\ref{fig:lossalae_final}.

\begin{figure}[H]
	\centering
	\includegraphics[width=\textwidth]{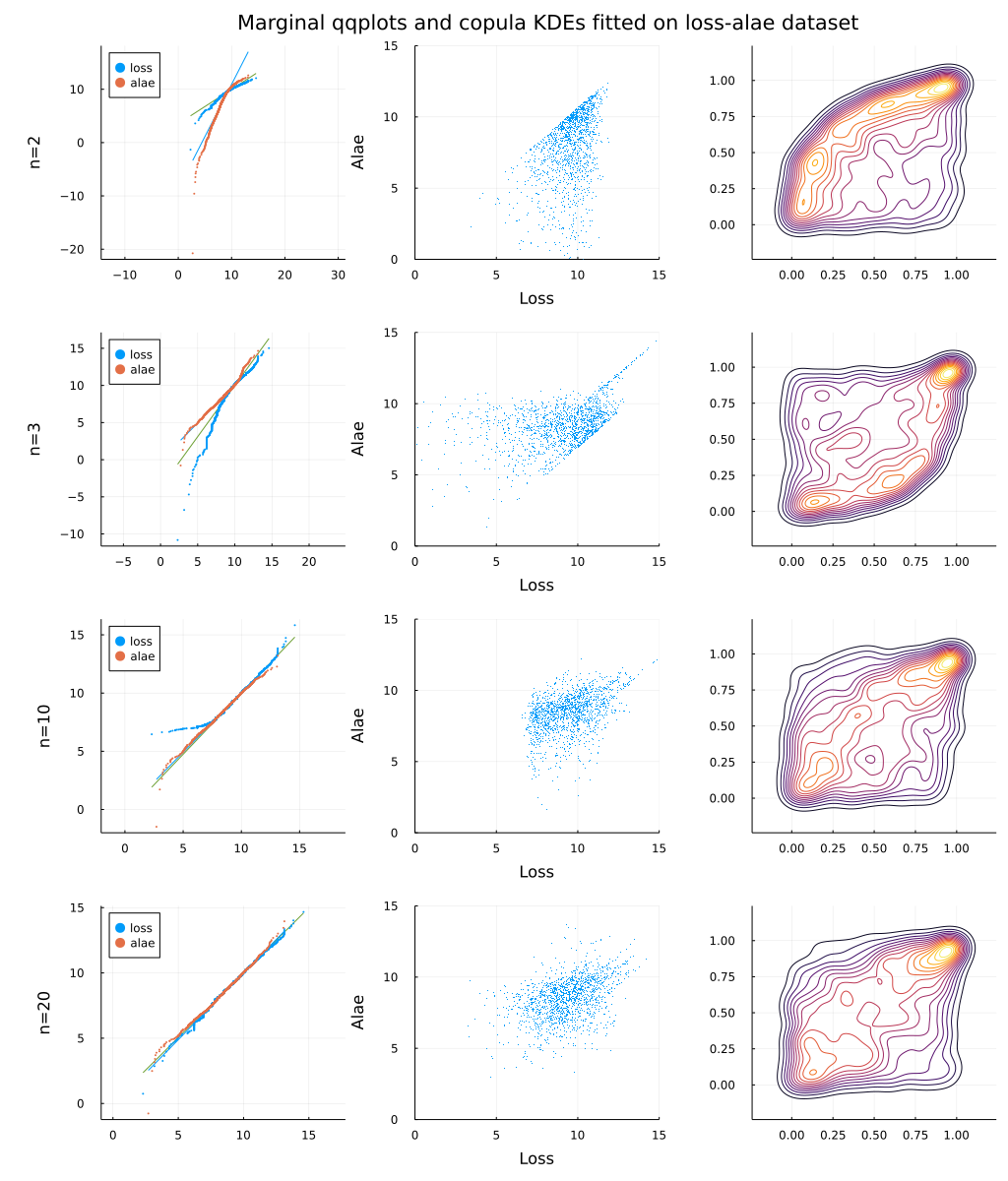}
	\caption{Results of the estimation on the Loss-Alae dataset.
  Each line correspond to a different number of gammas $n \in (2,3,10,20)$.
  The first column gives marginal quantile-quantile plots of Loss and Alae,
    the middle column gives a joint plot on log-scale,
    and the last column represents the dependence structure on the copula scale.}%
	\label{fig:lossalae_final}%
\end{figure}

We observe on the two first lines of Figure~\ref{fig:lossalae_final} a strange behavior for the dependence structure.
This is due to the fact that only one marginal has zeros in the scale matrices
  (see Table~\ref{tab:loss_alae_comparaison} in Appendix~\ref{sec:figs_and_tables}).  
Although unusual, 
  this kind of structure is quite expressive,
  and indeed for bigger values of $n$,
  where we convoluted several of them,
  we managed to capture the dependence structure of the data.
However,
  we have to wait until $n=20$ to obtain good fits on the marginals as well. 
Figure~\ref{fig:marginal_lossalae} reports more information about the final fits $n=20$,
  which shows the parametric smoothing on the marginals and the dependence structure.

\begin{figure}[H]
  \centering
  \includegraphics[width=\textwidth]{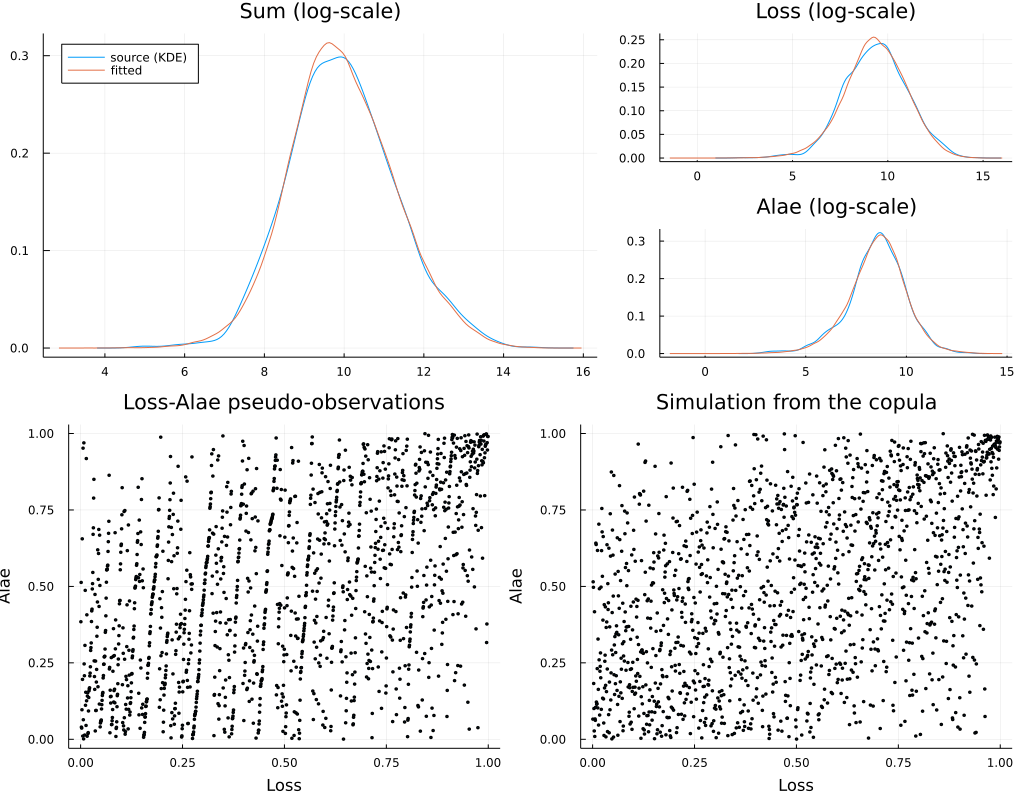}
  \caption{Results for the $n=20$ model. 
  Top left: density of the sum Loss + Alae. 
  Top right: marginal densities.
  Bottom: A comparison between pseudo-observations from a simulation (right) and the original dataset (left).}%
  \label{fig:marginal_lossalae}%
\end{figure}

Simulation from the obtained model is trivial, thanks to its additive structure.
Moreover, 
  the distribution of any linear combination $\langle \bm c, \bm X \rangle$ ($\bm c \ge \bm 0$) 
  is also given as a $\mathcal G_{1,n}(\alpha,\langle \bm c, \bm s \rangle)$.
This additive property is a computationally valuable feature of the model: 
  any result and potential code routine that computes statistics on $\mathcal G_{1,n}$ distribution 
  can be directly used for the marginals and for all linear combinations of them.
Finally, 
  Theorem 4.1.1 in ~\cite{bondesson1992} provides a neat mixture of gammas 
  representation for $\mathcal G_{1,n}$ distributions, 
  with easily computable mixture weights from parameters, 
  which can be useful for applications.

\section{Conclusion}\label{sec:conclusion}

Although the class of gamma convolutions has very appealing properties for the practitioner, 
  the estimation of gamma convolutions from empirical data had not been considered in the literature.
The closest algorithm we could find is the recent MFK algorithm,
  that we described in Subsection~\ref{sec:projection_furman},
  which projects densities from the $\mathcal G_1$ class to the $\mathcal G_{1,n}$ class. 
Unfortunately, extending this algorithm for general estimation purposes is hard if not impossible.

Using a well-chosen Laguerre basis,
  we constructed a series expansion for the density of a distribution in $\mathcal{G}_{d,n}$.
Coupled with a simple $L_2$ loss for the density, 
  we were able to project densities from $\mathcal{G}_d$,
  $d > 1$ onto $\mathcal G_{d,n}$,
  which was not possible,
  but also to project any density onto $\mathcal G_{d,n}$,
  which was not possible either.
Finally,
  through the same loss,
  we can estimate $\mathcal{G}_{d,n}$ distributions from empirical data for any reasonable $d,n$,
  which was also not treated in the literature.

Although the convergence result is not as strong as MFK's one,
  the numerical precision needs and computational speeds are far better.
We saw that given the same information,
  shifted moments of the random variable with high precision, 
  the performance of the two algorithms is overall the same,
  except maybe for small numbers of gammas where we beat MFK.
For both algorithms,
  Figure~\ref{fig:furman_ln_violin} also shows that the model's 
  goodness of fit seems to be decreasing for large numbers of gammas,
  and an adaptive estimator might be a suitable way out.

Maybe a better loss could be found for the estimation of these gamma convolutions.
Furthermore,
  automatic penalization of the model could be done through the number of gammas,
  but also through sparsity of the Thorin measure,
  according to Property~\ref{prop:special_thorin_measures},
  yielding more independent factors when possible.
Such a modification will also provide a clearer view at the additive risk factor structure.

The rate of the estimator will be explored further in future work, 
  using results from the parametric contrast estimators literature,
  as a referee pointed out.

Last but not least,
  many features are still to be discovered about multivariate Thorin classes.
The estimation of multiplicative structures could yield interesting results.
  Bondesson showed that the product of random variables in $\mathcal{G}_1$ is also in $\mathcal{G}_1$,
  but we still do not have a way of constructing the corresponding Thorin measure, 
  neither a corresponding result for the multivariate case.

\paragraph{Acknowledgment} We thank the referees for all the important remarks they made and all resulting improvement to the paper.

\begin{appendix}
  \section{Proofs of auxiliary results}\label{sec:proofs}

  \begin{proof}[Proof of Property~\ref{prop:special_thorin_measures}]

    Consider first atomic Thorin measures,
      so that $\bm X \sim \mathcal G_{d,n}(\bm \alpha,\bm s)$.
  
    For $(i)$ and $(ii)$,
      consider that there exists a matrix $\bm Y$ of gamma random variables
      $Y_{i,j} \sim \mathcal G_{1,1}(\alpha_j,s_{i,j})$
      such that $$X_i = Y_{i,1} + \ldots + Y_{i,n},$$
      where row vectors $\bm Y_{i,.}$ are independent 
      and column vectors $\bm Y_{.,j}$ have two separate groups of marginals:
      a first comonotonic group (with gamma marginals) corresponding to scales $s_{i,j}$ that are strictly positive
      and a second group (with identically 0 marginals) corresponding to null scales.
    Note that one of these two groups might be empty.
    From this expansion,
      we can easily deduce the structure of maximal supports for $\nu$ 
      that leads to independent or comonotonic random vectors $\bm X$.
  
    For $(iii)$, note that $D$ is the rank of the matrix $\bm s$, hence $D \le \min(n,d)$.
    If $D = d = n$,
      then $\bm X$ is by definition an invertible linear transformation of an independent,
      $d$-dimensional random vector,
      and therefore has a density in $L_2$ that we can express easily.
    If $D = d < n$,
      then $\bm X$ is the convolution of the previous case with something else,
      and therefore also has a density in $L_2$.
    The last case,
      $D <d$, 
      gives a random vector $X$ that is a linear transformation of a set of $D < d$ independent variables,
      which is therefore clearly not absolutely continuous as claimed.
  
    Consider non-atomic Thorin measures as limits of the previous case.
  \end{proof}

	\begin{proof}[Proof of Example~\ref{prop:curious_dist}]
		For the first result,
		  note that $\nu$ satisfies the integration constraints given by Definition \ref{univ_ggc_def}.
		Then,
		  simply compute the cumulant generating function as:
		\begin{align*}
		K(-t) &= -\int\limits_{0}^{\infty} \ln\left(1+rt\right) \mathbb 1_{r \in [0,1]} \,d r
      = -\int\limits_{0}^{1} \ln\left(1+rt\right)  \,d r\\ 
      &= -\frac{1}{t}\int\limits_{1}^{1+t} \ln(x) \,d x
      = 1 - \frac{1+t}{t}\ln(1+t)\text{.}
		\end{align*}
    The second result comes from a discretization of the Thorin measure,
		  from continuously uniform on $[0,1]$ to uniform on $\left\{\frac{j}{n+1},j \in \{1,...,n\} \right\}$.
    Note that the cumulant generating functions converges on its domain of definition,
      that includes a real interval,
      hence the convergence in distribution. 
	\end{proof}

\begin{proof}[Proof of Example~\ref{ex:d1-inversion}]

  The bijection is given through the following equations:
  $$a_{\bm k}(\alpha,\bm s) = \sqrt{2}^d \sum\limits_{\bm \ell \le \bm k} \binom{\bm k}{\bm \ell} 
    \frac{(-2\bm s)^{\bm \ell}}{\bm \ell !} \frac{\Gamma\left(\alpha + \lvert \bm \ell \rvert\right)}{\Gamma\left(\alpha\right)} 
    \left(1 + \lvert \bm s \rvert\right)^{-\alpha - \lvert \bm \ell \rvert},$$
	on one hand,
  and
	\begin{align*}
		\alpha &= \left\{\frac{1}{c_2} - \frac{1}{\ln\left(c_1\right)} 
    W_0\left(\frac{\ln\left(c_1\right)}{c_2}e^{\frac{\ln\left(c_1\right)}{c_2}}\right)\right\}\\
		s_i &=\frac{a_{\bm 0} - a_{\bm 1(i)}}{2\alpha a_0} \text{ for all $i \in 1,..,d$}
		\end{align*}
		reciprocally,
  where:
		\begin{itemize}
			\item $\bm 1(i)$ is the $d$-variate vector with  $j^{\text{th}}$ component equal to $\mathbb 1_{i=j}$, for all $j$.
			\item $c_1 = a_{\bm 0}\sqrt{2}^{-d}$ and $c_2 = \frac{d}{2} - \frac{1}{2a_{\bm 0}}\sum_{i=1}^d a_{\bm 1(i)}$
			\item $W_0$ is the zeroth branch of the Lambert function $W$ defined by: $$y = W(x) \iff x= ye^y.$$
		\end{itemize}

	To prove the first part,
    note that there exists $X_0 \sim \mathcal G_{1,1}(\alpha,1)$ such that 
    $\bm X = \scalarprod{\bm s}{\left(X_0,...,X_0\right)}$.
  Therefore,
    denoting by $M$ the mgf of $X_0$,
    \begin{align*}
      \mu_{\bm k,-\bm 1} = \mathbb E\left(\bm X^{\bm k} e^{-\scalarprod{\bm 1}{\bm X}}\right)
    &= \bm s^{\bm k}\mathbb E\left(X_0^{\lvert \bm k \rvert} e^{-\lvert \bm s \rvert X_0}\right)\\
    &= \bm s^{\bm k} M^{\left(\lvert \bm k \rvert\right)}(-\lvert \bm s \rvert)\\
    &= \bm s^{\bm k} \frac{\Gamma\left(\alpha + \lvert \bm k \rvert\right)}{\Gamma\left(\alpha\right)} (1 + \lvert \bm s \rvert)^{-\alpha - \lvert \bm k \rvert}.
    \end{align*}
	
	Therefore,
    the Laguerre coefficients $a_{\bm k}\left(\alpha,\bm s\right)$ of $\bm X$ are given by:
    $$a_{\bm k}(\alpha,\bm s) = \sqrt{2}^d \sum\limits_{\bm \ell \le \bm k} \binom{\bm k}{\bm \ell} \frac{(-2\bm s)^{\bm \ell}}{\bm \ell !} 
      \frac{\Gamma\left(\alpha + \lvert \bm \ell \rvert\right)}{\Gamma\left(\alpha\right)} 
      \left(1 + \lvert \bm s \rvert\right)^{-\alpha - \lvert \bm \ell \rvert}.$$

	Then,
    for the reverse assertion,
    consider the expression of the first Laguerre coefficients of $\bm X$:
    $$a_{\bm 0} = \sqrt{2}^d (1 - \lvert \bm s \rvert)\text{ and }a_{\bm 1(i)} 
      = a_{\bm 0}\left(1 - 2\alpha \frac{s_i}{1 + \lvert \bm s\rvert}\right).$$
  Use the simplex transformation $\bm x = \frac{\bm s}{1 + \lvert\bm s\rvert}$ 
    to solve for $\bm s$ as a function of $\alpha$ in the second equation,
    and inject in the first.
	We obtain an equation of the form $ae^x + bx +c$ which is solved by the zeroth branch of the Lambert W function.
\end{proof}

\begin{proof}[Proof of Lemma~\ref{lem:mobius}]
  Point $(i)$ and $(ii)$ are obvious. 
  To prove the 3 remaining points,
    we use the fact that $h$ is a Möbius transform,
    see~\cite{arnold2008},
    and hence maps generalized circles to generalized circles.
  Furthermore,
    $h$ maps reals to reals,
    and we obtain $c(b)$ and $r(b)$ from the expression of $h(b)$ and $h(-b)$.
  Last, by checking if the disk with center $c(b)$ and radius $r(b)$ contains $h(0) = -1$, we conclude. 
\end{proof}

\begin{proof}[Proof of Example~\ref{ex:wb}]
  The three first cases are deduced directly from Definition~\ref{def:eps_wb} and Property~\ref{prop:epsilon_wb_2}. 
  For point $(iv)$,
    recall that for two gamma convolutions with Thorin measures $\nu_1$ and $\nu_2$, the Thorin measure of the sum is:
    $$\nu(\bm s) = \nu_1(\bm s)+\nu_2(\bm s).$$
  Now, according to the reading of Property~\ref{prop:epsilon_wb_2}, 
    point $(iv)$,
    no set of $p <d$ rays $S_{\bm c_1},...,S_{\bm c_p}$ concentrate more than half 
    of $\nu_1$ weights and more than half of $\nu_2$ weights.
  Hence,
    no set of $p<d$ rays can concentrate more than half of $\nu$ weights.
  Therefore, the class of w.b. gamma convolutions is closed w.r.t finite convolutions.
\end{proof}

  \section{Auxiliary figures and tables}\label{sec:figs_and_tables}

\begin{figure}[H]
	\centering
	\includegraphics[width=\textwidth]{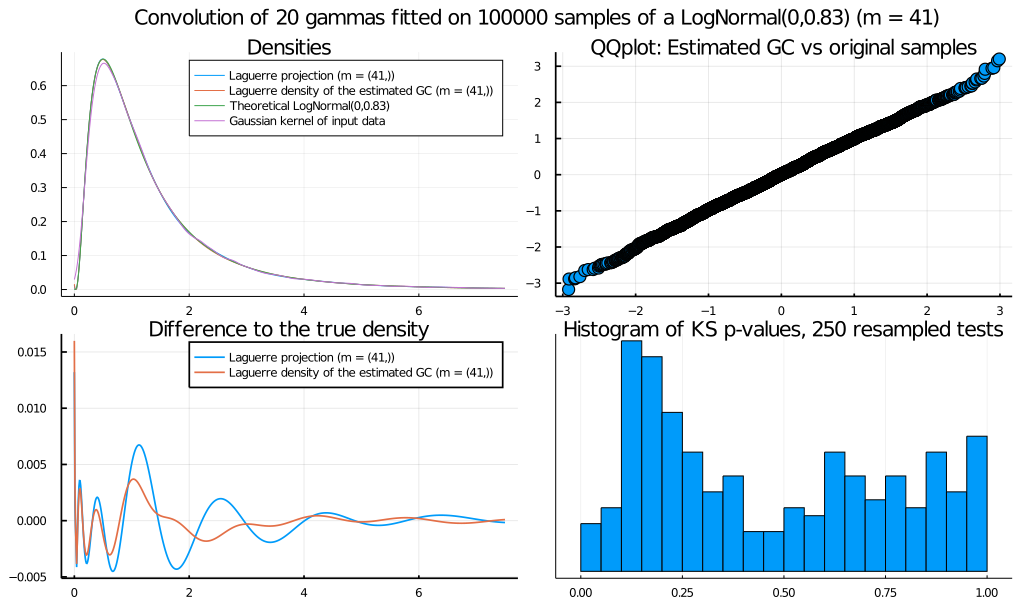}
	\caption{Log-Normal results with $20$ gammas.
			 Same legend as Figure~\ref{fig:log-Normal10}.}%
	\label{fig:log-Normal20}%
\end{figure}

\begin{figure}[H]
	\centering
	\includegraphics[width=\textwidth]{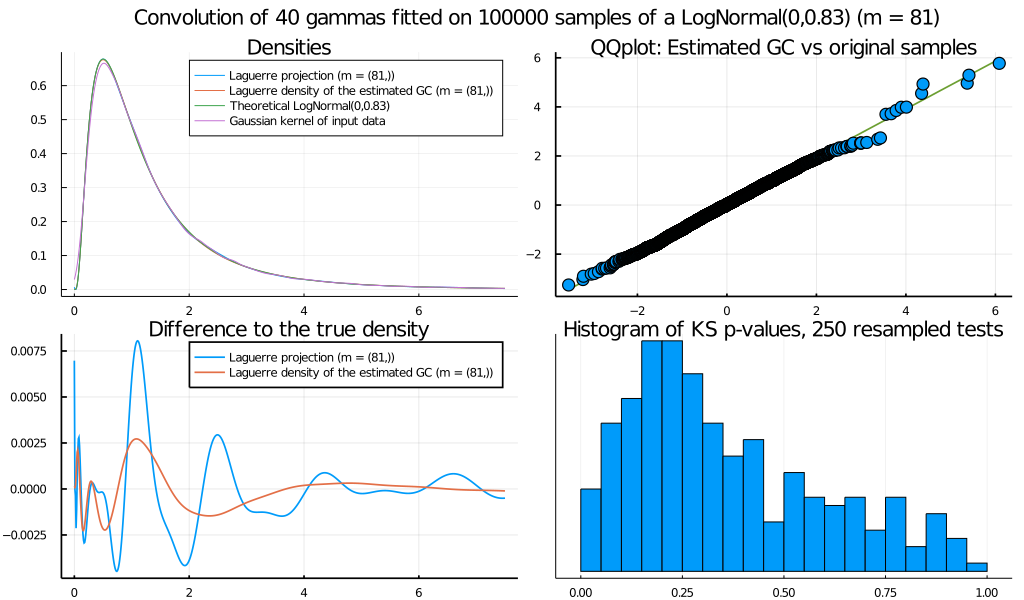}
	\caption{Log-Normal results with $40$ gammas.
			 Same legend as Figure~\ref{fig:log-Normal10}.}%
	\label{fig:log-Normal40}%
\end{figure}

  \begin{DIFnomarkup}
    \begin{table}[H]
    \scriptsize
    \caption{\label{tab:loss_alae_comparaison}Estimated Thorin measures on Loss-Alae dataset, for several number of atoms. 
    Scales below $10^{-10}$ are considered to be $0$.}
    \centering
    \begin{tabu} to \linewidth {>{\raggedright}X>{\raggedleft}X>{\raggedleft}X>{\raggedleft}X}
    \toprule
    & \(\hat{\bm \alpha}\) & \(\hat{\bm s}_{.,1}\) & \(\hat{\bm s}_{.,2}\)\\
    \midrule
    \addlinespace[0.3em]
    \hspace{1em}\textbf{$n=2$} & $8.0530e-01$ & $1.7859e+04$ & $0$\\
    \hspace{1em}               & $3.7345e-01$ & $1.9294e+04$ & $3.9144e+04$\\
    \addlinespace[0.3em]
    \hspace{1em}\textbf{$n=3$} & $6.2517e-01$ & $0$ & $9.4232e+03$\\
    \hspace{1em}               & $4.0146e-01$ & $5.9577e+04$ & $2.8870e+03$\\
    \hspace{1em}               & $2.8491e-02$ & $7.4721e+05$ & $5.1605e+05$\\
    \addlinespace[0.3em]
    \hspace{1em}\textbf{$n=4$} & $4.2930e-01$ & $8.9338e+02$ & $4.5100e+02$\\
    \hspace{1em}               & $2.7618e-01$ & $7.5227e+03$ & $2.1992e+04$\\
    \hspace{1em}               & $1.6763e-01$ & $1.1016e+05$ & $2.1315e+03$\\
    \hspace{1em}               & $2.9678e-02$ & $7.7123e+05$ & $1.6315e+05$\\
    \addlinespace[0.3em]
    \hspace{1em}\textbf{$n=5$} & $4.4402e-01$ & $1.9749e+02$ & $1.2632e+03$\\
    \hspace{1em}               & $3.8222e-01$ & $2.5404e+04$ & $3.2267e+02$\\
    \hspace{1em}               & $3.0692e-01$ & $6.8942e+03$ & $1.8928e+04$\\
    \hspace{1em}               & $4.9173e-02$ & $3.0236e+05$ & $4.9963e+03$\\
    \hspace{1em}               & $4.3674e-02$ & $4.0133e+05$ & $7.5906e+04$\\
    \addlinespace[0.3em]
    \hspace{1em}\textbf{$n=10$} & $5.8686e-01$ & $6.7317e+01$ & $0$\\
    \hspace{1em}               & $4.1458e-01$ & $0$ & $4.2265e+03$\\
    \hspace{1em}               & $3.8097e-01$ & $0$ & $4.2032e+03$\\
    \hspace{1em}               & $1.8363e-01$ & $3.0710e+04$ & $0$\\
    \hspace{1em}               & $1.4472e-01$ & $2.0463e+04$ & $1.6083e+04$\\
    \hspace{1em}               & $5.8809e-02$ & $3.0346e+04$ & $0$\\
    \hspace{1em}               & $5.3588e-02$ & $5.8625e+03$ & $5.0686e+04$\\
    \hspace{1em}               & $4.9132e-02$ & $1.8103e+05$ & $1.6242e+03$\\
    \hspace{1em}               & $3.9180e-02$ & $1.9806e+05$ & $4.0709e+04$\\
    \hspace{1em}               & $1.8969e-02$ & $9.4263e+05$ & $6.6635e+04$\\
    \addlinespace[0.3em]
    \hspace{1em}\textbf{$n=20$} & $5.4500e-01$ & $0$ & $4.4628e+03$\\
    \hspace{1em}               & $2.3101e-01$ & $3.4445e+03$ & $0$\\
    \hspace{1em}               & $1.8010e-01$ & $0$  & $4.4628e+03$\\
    \hspace{1em}               & $1.5530e-01$ & $3.2021e+04$ & $5.7533e+02$\\
    \hspace{1em}               & $1.5374e-01$ & $3.4445e+03$ & $0$ \\
    \hspace{1em}               & $1.0508e-01$ & $3.4445e+03$ & $0$ \\
    \hspace{1em}               & $9.4038e-02$ & $3.4445e+03$ & $0$ \\
    \hspace{1em}               & $8.6941e-02$ & $9.6571e+04$ & $7.8920e+01$\\
    \hspace{1em}               & $8.3410e-02$ & $1.3871e+04$ & $2.0096e+04$\\
    \hspace{1em}               & $5.0148e-02$ & $3.4445e+03$ & $0$ \\
    \hspace{1em}               & $4.0136e-02$ & $1.3871e+04$ & $2.0096e+04$\\
    \hspace{1em}               & $3.4315e-02$ & $0$  & $4.4628e+03$\\
    \hspace{1em}               & $3.3641e-02$ & $2.3235e+05$ & $3.6468e+04$\\
    \hspace{1em}               & $3.1375e-02$ & $1.3289e+05$ & $5.2870e+04$\\
    \hspace{1em}               & $3.0619e-02$ & $0$  & $4.4628e+03$\\
    \hspace{1em}               & $1.4615e-02$ & $1.3871e+04$ & $2.0096e+04$\\
    \hspace{1em}               & $1.4270e-02$ & $6.5957e+05$ & $3.1213e+04$\\
    \hspace{1em}               & $8.5642e-03$ & $3.9661e+03$ & $3.3040e+05$\\
    \hspace{1em}               & $4.6750e-03$ & $4.5138e+04$ & $2.3531e+05$\\
    \hspace{1em}               & $1.2388e-13$ & $4.1893e+03$ & $0$\\
      \bottomrule
      \end{tabu}
    \end{table}
    \end{DIFnomarkup}
    
\end{appendix}



\bibliographystyle{imsart-number} 
\bibliography{rfm}       


\end{document}